\documentclass[12pt]{amsart}

\usepackage[textwidth=16cm, textheight=22cm]{geometry}

\usepackage{graphicx}
\usepackage{frenchineq}
\usepackage{hyperref}
\newcommand{\arxiv}[1]{\href{http://www.arXiv.org/abs/#1}{arXiv:#1}}
\newcommand{\doi}[1]{\href{http://dx.doi.org/#1}{doi:#1}}
\usepackage{graphicx}
\usepackage{latexsym}
\usepackage{amsmath}
\usepackage{amssymb}
\usepackage{tikz}
\newcommand {\eproof}{\space{\ \vbox{\hrule\hbox{\vrule height1.3ex\hskip0.8ex\vrule}\hrule}}\vskip 0.3cm \par}
\newcommand{\mysup}{\mathop{\text{\Large$\vee$}}}
\newcommand{\myinf}{\mathop{\text{\Large$\wedge$}}}

\newcommand{\cV}{{\mathcal V}}
\newcommand{\RR}{\mathbb{R}_{\max}}
\newcommand{\Rmax}{\mathbb{R}_{\max}}
\newcommand{\RRbar}{\overline{\mathbb{R}}_{\max}}
\newcommand{\RRm}{\RR^m}
\newcommand{\RRn}{\RR^n}
\newcommand{\RRmm}{\RR^{m\times n}}
\newcommand{\R}{\mathbb{R}}

\newcommand{\spec}{\operatorname{spec}}
\newcommand{\diez}{\sharp}
\newcommand{\Rp}{\R_+}

\newcommand{\supp}{\operatorname{supp}}
\newcommand{\spann}{\operatorname{span}}
\newcommand{\bzero}{-\infty}
\newcommand{\specf}{s}

\newcommand{\hilbdist}{d_{\operatorname{H}}}
\newcommand{\chebdist}{d_{\infty}}
\newcommand{\rad}{r}
\newcommand{\MPG}{\operatorname{MPG}}
\newcommand{\mymult}{\otimes}
\newcommand\prslash{\ifx\@currsize\normalsize
{\mathchoice
{\mbox{\raisebox{0.2ex}{$\scriptstyle\circ$}\kern-0.96ex$/$}}
{\mbox{\raisebox{0.2ex}{$\scriptstyle\circ$}\kern-0.96ex$/$}}%
{\mbox{\raisebox{0.14ex}{$\scriptscriptstyle\circ$}\kern-0.70ex%
${\scriptstyle/}$}}%
{\mbox{\raisebox{0.14ex}{$\scriptscriptstyle\circ$}\kern-0.70ex%
${\scriptstyle/}$}}}%
\else\ifx\@currsize\large\mbox{\raisebox{0.2ex}{$\scriptstyle\circ$}\kern-0.93ex$/$}
\else\ifx\@currsize\small\mbox{\raisebox{0.2ex}{$\scriptstyle\circ$}\kern-0.97ex$/$}
\else\mbox{\raisebox{0.2ex}{$\scriptstyle\circ$}\kern-0.96ex$/$} \fi\fi\fi}
\newcommand{\mydiv}{\prslash}
\newcommand{\mytimes}{}
\newcommand{\maxlength}{\kappa}

\newtheorem{theorem}{Theorem}
\newtheorem{lemma}{Lemma}
\newtheorem{definition}{Definition}
\newtheorem{proposition}{Proposition}

\newtheorem{corollary}{Corollary}
\theoremstyle{remark}
\newtheorem{example}{Example}
\newtheorem{remark}{Remark}
\begin{document}
\title{The level set method for the two-sided max-plus eigenproblem}
\thanks{The first author was partially supported by the Arpege  programme of the
French National Agency of Research (ANR), project ``ASOPT'', number ANR-08-SEGI-005
and by the Digiteo project DIM08 ``PASO'' number 3389.
The second author was supported by the EPSRC grant RRAH12809 and the RFBR grant 08-01-00601}
\author{St\'ephane Gaubert}
\address{INRIA and Centre de Math\'ematiques Appliqu\'ees,\'Ecole Polytechnique.
Postal address: CMAP, \'Ecole Polytechnique, 91128 Palaiseau C\'edex, France.}
\email{Stephane.Gaubert@inria.fr}

\author{Serge\u{\i} Sergeev}
\address{University of Birmingham,
School of Mathematics, Watson Building, Edgbaston B15 2TT, UK}
\email{sergeevs@maths.bham.ac.uk}
\date{\today}
\catcode`\@=11
\@namedef{subjclassname@2010}{%
  \textup{2010} Mathematics Subject Classification}
\catcode`\@=12

\begin{abstract}
We consider the max-plus analogue of the eigenproblem for matrix pencils,
$A\mymult x=\lambda\mymult B\mymult x$. We show that the spectrum of $(A,B)$ (i.e., the set of
possible values of $\lambda$), which is a finite union of intervals, can be computed in
pseudo-polynomial number of operations, by a (pseudo-polynomial) number of calls to an oracle that computes the value of a mean payoff game.
The proof relies on the introduction of a spectral function, which we interpret in terms of the
least Chebyshev distance between $A\mymult x$ and $\lambda\mymult B\mymult x$. The spectrum is obtained
as the zero level set of this function.

\end{abstract}

\maketitle

\section{Introduction}

\subsection{Motivations and general information}

Max-plus algebra is the analogue of linear algebra
developed over the max-plus semiring which is
the set $\RR=\R\cup\{-\infty\}$ equipped with the operations
of ``addition'' $a\oplus b:=a\vee b=\max(a,b)$ and
``multiplication'' $a\mymult b:=a+b$.
The zero of this semiring is $-\infty$, and the unit
of this semiring is $0$. Note that $a^{-1}$ in max-plus is the same as $-a$
in the conventional notation.
The operations of the
semiring are extended to matrices and vectors over
$\RR$. That is if $A=(a_{ij})$,
$B=(b_{ij})$ and $C=(c_{ij})$ are matrices of compatible sizes with
entries from $\RR$, 
we write $C=A\vee B$ if $c_{ij}=a_{ij}\vee b_{ij}$ for all $i,j$ and $%
C=A\mymult B$ if $c_{ij}=\mysup_k (a_{ik}+b_{kj})$ for all $i,j$.




We investigate the two-sided eigenproblem in max-plus algebra: for
two matrices $A,B\in\RR^{m\times n}$, find scalars $\lambda\in\RR$
called {\em eigenvalues}
and vectors $x\in\RRn$ called
{\em eigenvectors}, with at least one component not equal to 
$-\infty$, such that
\begin{equation}
\label{twosided-eig}
A\mymult x=\lambda\mymult B\mymult x,
\end{equation}
where the operations have max-plus algebraic sense.
In the conventional notation this reads
\begin{equation}
\label{twosided-eig-conventional}
\max_{j=1}^n (a_{ij}+x_j)=\lambda+\max_{j=1}^n (b_{ij}+x_j),\quad\text{for}\; i=1,\ldots,m.
\end{equation}
The set of eigenvalues will be called the {\em spectrum}
of $(A,B)$ and denoted by $\spec(A,B)$.

When $B$ is the max-plus identity matrix $I$ (all diagonal entries equal $0$ and
all off-diagonal entries equal $-\infty$), problem \eqref{twosided-eig} is the max-plus
spectral problem. The latter spectral problem, as well as
its continuous extension for max-plus linear operators,
is of fundamental importance for
a wide class of problems in discrete event systems theory,
dynamic programming, optimal control and mathematical physics~\cite{BCOQ,HOW:05,KM:97}.

Problem~\eqref{twosided-eig} is related
to the Perron-Frobenius theory for the two-sided eigenproblem
in the conventional linear algebra,
as studied in~\cite{McD+98,MNV-08}. When both
matrices are nonnegative and depend on a large parameter,
it can be shown following
the lines of \cite[Theorem 1]{ABG-98}
that the asymptotics of an eigenvalue with nonnegative
eigenvector is controlled by an eigenvalue
of \eqref{twosided-eig}.
This argument calls for the development of two-sided
analogue of the tropical eigenvalue perturbation theory
presented in \cite{ABG-04,ABG-06pencils}.

A specific motivation to study the two-sided max-plus eigenproblem
arises from discrete event systems. In particular, systems of the
form $A\mymult x=B\mymult x$ or $A\mymult x\leq B\mymult x$ appear in scheduling.
Indeed, when $\lambda=0$, the system of constraints~\eqref{twosided-eig-conventional} can be interpreted in terms of {\em rendez-vous}. Here, $x_j$
represents the starting time of a task $j$ (for instance,
the availability of a part in a manufacturing system). The
expression $\max_{j=1}^n (a_{ij}+x_j)$ represents the earliest
completion time of a task which needs at least $a_{ij}$ time units
to be completed after task $j$ started. Thus, the system 
$A\mymult x=B\mymult x$ requires to find starting times such that two different
sets of tasks are completed at the earliest exactly at the same
times. In many situations, such systems cannot be solved exactly, and a natural 
idea is to calculate the minimal
Chebyshev distance between $A\otimes x$ and $B\otimes x$. Theorem 4 below
determines this minimal distance.
It may be also of interest to solve perturbed problems like $A\mymult x=\lambda\mymult
B\mymult x$, as in~\eqref{twosided-eig}. Such problems express {\em no wait}
type constraints. Indeed, $y:=A\mymult x$ and $z:=B\mymult x$ may be thought
of as the outputs of two different systems $A$ and $B$, with a common
input $x$. The time offsets between output events are represented
by the differences $y_i-y_j$, for all $i,j$, where the difference
is understood in the usual algebra
(these quantities belong to the ``second order'' max-plus theory, see
e.g.~\cite{maxplus91b}). No wait constraint may require that $y_i-y_j$
take prescribed values, for some pair $(i,j)$. The condition that
$y=\lambda\mymult z$, i.e., $y=\lambda+z$, for some $\lambda\in\mathbb{R}$,
means precisely that the time offsets
are the same for the two outputs $y$ and $z$. Hence, an input $x$ solving
$A\mymult x=\lambda\mymult B\mymult x$ has the property of making $A$ and $B$
indistinguishable from the point of view of no wait output constraints. An example
of such a situation is demonstrated on Figure~\ref{fig-new} in Subsection~\ref{ss:bounds}.

Problems of a related nature, regarding the time separation
between events, arose for instance in the work of Burns,
Hulgaard,  Amon, and Borriello~\cite{burns95}, following the work of
Burns on the checking of asynchronous digital
circuits~\cite{Bur:91}. Moreover, systems of the form $A\mymult x\leq B\mymult x$ 
represent scheduling problems with
both AND and OR precedence constraints, studied by M\"ohring,
Skutella, and Stork~\cite{mohring}.


Similar motivations led to the study of min-max functions by Olsder~\cite{Ols-91} and Gunawardena~\cite{Gun-94}. Such functions can be
written as finite infima of max-plus linear maps, or finite suprema of min-plus linear maps. They also arise as dynamic
programming operators of zero-sum deterministic games.
In particular, the fixed points and invariant
halflines of min-max functions studied in \cite{CGG-99,DG-06} can be
also used to compute values of zero-sum deterministic games with
mean payoff \cite{DG-06,ZP-96}.  A correspondence between the computation of
the value of mean payoff games and two-sided linear systems in max-plus
algebra has been established in~\cite{AGG-10}; we shall exploit here
the same correspondence, although in different guises.

In max-plus algebra, a special form of min-max functions appears in
Cun\-ing\-ha\-me-Green \cite{CG:79}, under the name of
$AA^*$-products.
The same functions
appear as nonlinear projectors on max-plus cones
playing essential role
in the max-plus analogue of Hahn-Banach
theorem \cite{CGQS-05,LMS-01}.
The compositions of nonlinear projectors are more general min-max
functions, and they appear when one approaches two-sided systems
$A\mymult x=B\mymult  y$ and $A\mymult x=B\mymult x$ \cite{CGB-03}, and
intersections of max-plus cones \cite{GS-08,Ser-09-inLS}. It is immediate
to see that~\eqref{twosided-eig} is a parametric version of $A\mymult
x=B\mymult x$.

In max-plus algebra, partial results for Problem~\eqref{twosided-eig} have been obtained by
Binding and Volkmer \cite{BV-07}, and
Cuninghame-Green and Butkovi\v{c} \cite{But:10,CGB-08}.
In particular, Cuninghame-Green and Butkovi\v{c} \cite{But:10,CGB-08}
give an interval bound on the spectrum of \eqref{twosided-eig}
in the case where the entries of both matrices are real.
Besides that, both papers treat interesting special cases,
for instance when $A$ and $B$ square, or
one of them is a multiple of the other.

The spectrum of~\eqref{twosided-eig} is generally a collection of intervals on the real line. 
By means of projection, this follows from a result of 
De Schutter and De Moor~\cite{SM-96} that solution set to the system of max-plus
(in)equalities is a union of convex polyhedra. Note that the approach of~\cite{SM-96}, related to Develin-Sturmfels cellular decomposition~\cite{DS-04}, can be also used for solving
$A\otimes x=\lambda\otimes B\otimes x$ and more general problems of max-plus linear algebra.

\subsection{Contents of the paper}

In the present paper, we first show that \eqref{twosided-eig} can be
viewed as a fixed-point problem for a family of parametric min-max
functions $h_{\lambda}$. Based on this observation, we introduce a
spectral function $s(\lambda)$ of \eqref{twosided-eig}, defined as
the spectral radius of $h_{\lambda}$. The zero level set of
$s(\lambda)$ is precisely $\spec(A,B)$. More generally, $s(\lambda)$
has a natural geometric sense, being equal to the inverse of the least Chebyshev
distance between $A\mymult x$ and $\lambda\mymult B\mymult x$.

The function $s(\lambda)$ is piecewise-affine and Lipschitz
continuous, and it has an affine asymptotics at large and small
$\lambda$. In an important special case when none of the matrices $A$ and $B$ have 
an identically $-\infty$ column, the asymptotics is just
$\lambda+\alpha_1$ at small $\lambda$, and $-\lambda+\alpha_2$ at
large $\lambda$, in the conventional arithmetics. We also give bounds on the spectrum of two-sided
eigenproblem, which improve and generalize the bound
of Cuninghame-Green and Butkovi\v{c} \cite{But:10,CGB-08}.
In the case when the entries of $A$
and $B$ are integer or $-\infty$, this allows us to show that all
affine pieces of $s(\lambda)$ can be identified in a
pseudopolynomial number of calls to an oracle which identifies $s(\lambda)$ at a given point.
Importantly, $s(\lambda)$ can be interpreted as the greatest value of the associated parametric
mean-payoff game and it can be computed by
the policy iteration algorithm of \cite{CGG-99,DG-06}, as well as by the value iteration
of Zwick and Paterson~\cite{ZP-96} or the subexponential method of Bj\"{o}rklund and Vorobyov~\cite{bjorklund}. This leads to a procedure for computing the whole spectrum of
\eqref{twosided-eig}. To our knowledge, no such general algorithm for
computing the whole spectrum of \eqref{twosided-eig}
was known previously. We also believe that the level set
method used here, relying on the introduction of the spectral
function, is of independent interest and may have other applications.

In some cases the spectral function can be computed analytically. In
particular, we will consider an example of \cite{Ser-note-10}, where
it is shown that any finite system of intervals and points on the
real line can be represented as the spectrum of \eqref{twosided-eig}.

The paper is organized as follows. In the remaining subsection of Introduction
we explain 
the notation used in the
rest of the paper. In Section 2 we consider
two-sided systems $A\mymult x=B\mymult y$ and $A\mymult x=B\mymult x$.  We
relate the systems $A\mymult x=B\mymult x$ to certain min-max functions
and show that the spectral radii of these functions are equal to the inverse of the
least Chebyshev distance between $A\mymult x$ and $B\mymult x$. In
Section 3, we introduce the spectral function of two-sided
eigenproblem as the spectral radius of a natural parametric
extension of the min-max functions studied in Section 2. We give
bounds on the spectrum of two-sided eigenproblem and investigate the
asymptotics of $s(\lambda)$. We reconstruct the spectral function
and hence the whole spectrum in a pseudopolynomial number of calls
to the mean-payoff game oracle. 

\subsection{Notation}

For the sake of simplicity, the sign $\mymult$ will be usually omitted in the remaining part of the paper, or even replaced with $+$ if scalars are involved. In particular we write $Ax$ for $A\mymult x$
and $\lambda+ x$ for $\lambda\mymult x$, where $A\in\RR^{m\times n}$ (a matrix), $x\in\RRn$ (a vector) and $\lambda\in\RR$ (a scalar).
Moreover in the remaining part of the paper
we will prefer 
conventional arithmetic notation: the four arithmetic operations $a+b$, $a-b$, $ab$ and $a/b$ on the
set of real numbers (scalars) will have their usual meaning. However we often use $\vee$ for max and $\wedge$ for $\min$ (also componentwise). The actions of max-plus linear
operators, their min-plus linear residuations and nonlinear projectors onto max-plus cones (defined below in Subsection~\ref{ss:prel}),  which will appear as $Ax$, $A^{\sharp}y$, $P_A y$ etc., should not be confused with any conventional linear operator. The notations like $A^{\sharp}B$ or $P_AP_B$ should be understood as compositions of the corresponding operators rather than any kind of matrix multiplication between them  (in the case of $A^{\sharp}B$ above, $B$ is max-plus linear and $A^{\sharp}$ is min-plus linear). 

Such notation is implied by the methodology of the paper: we consider a problem of max-plus algebra, using nonlinear Perron-Frobenius theory and elementary analysis of piecewise-affine functions, where the max-plus arithmetic notation is not required or is inconvenient. The max-plus matrix product notation $\mymult$, especially when mixed with the dual min-plus $\mymult'$, is also
inconvenient when several compositions of such operators appear in the same formula or equation.

The notation used in our main subsections is not new. Being different from Baccelli et al.~\cite{BCOQ} or Cuninghame-Green~\cite{CG:79}, it basically follows Gaubert and 
Gunawardena~\cite{GG-98}. 

\section{Two-sided systems and min-max functions}

\subsection{Max-plus linear systems and nonlinear projectors}
\label{ss:prel} Consider the $m$-fold Cartesian product $\RRm$
equipped with operations of taking supremum  $u\vee v$ and scalar
``multiplication'' (i.e., addition) $\lambda\mymult v:=\lambda+v$.
This structure is an example of semimodule over the semiring $\RR$
defined in the introduction. The subsets of $\RRm$ closed under
these two operations are its subsemimodules. We will call them {\em
max-plus cones} or just {\em cones}, by abuse of language. Indeed,
there are important analogies and links between max-plus cones and
convex cones \cite{CGQS-05,DS-04,GK-09,Ser-09-inLS}. We also need
the operation of taking infimum which we denote by $\inf$ or $\wedge$.


With a max-plus cone $\cV\subseteq\RRm$ we can associate an operator $P_{\cV}$
defined by its action
\begin{equation}
P_{\cV} z=\mysup\{y\in\cV\mid y\leq z\}.
\end{equation}
Consider the case where $\cV\subseteq\RRm$ is {\em generated} by a set
$S\in\RRm$, which means that it is the set of bounded max-plus linear
combinations
\begin{equation}
v=\mysup_{y\in S} \lambda_y+y.
\end{equation}
In this case
\begin{equation}
\label{projector}
\begin{split}
P_{\cV} z&=\mysup_{y\in S} z\mydiv y + y,\ \text{where}\\
z\mydiv y&=\max\{\gamma\mid\gamma+y\leq z\}=
\myinf_{j\in\supp(y)} (z_j- y_j)=
\myinf_{j=1}^m (z_j-y_j),
\end{split}
\end{equation}
with the convention $(-\infty)+(+\infty)=+\infty$.
Here and in the sequel $\supp(y):=\{i\mid y_i\neq -\infty\}$
denotes the support of $y$.
Note that $z\mydiv y=\infty$ if and only
if $y=\bzero$.

Further we are interested only in the case when $\cV$
is finitely generated. Let $S=\{y^1,\ldots,y^n\}$, and
$T_i$ denote
the set of indices where the minimum in $z\mydiv y^i$ is attained.
The following result is classical.

\begin{proposition}[\cite{BCOQ,But-03,HOW:05}]
\label{p:ax=b}
Let a cone $\cV\subseteq\RRm$ be generated by $y^1,\ldots, y^n$
and let $z\in\RRm$. The following statements are equivalent.
\begin{itemize}
\item[1.] $z\in\cV$.
\item[2.] $P_{\cV} z=z$.
\item[3.] $\bigcup_{i=1}^m T_i=\supp z$.
\end{itemize}
\end{proposition}
We note that the set covering condition 3. has been generalized to
the case of Galois connections \cite{AGK}.

By this proposition, the operator $P_{\cV}$ is a projector onto
$\cV$. It is an isotonic and $+$-homogeneous operator, meaning
that $z^1\leq z^2$ implies $P_{\cV} z^1\leq P_{\cV} z^2$, and that
$P_{\cV}(\lambda  + z)=\lambda  + P_{\cV} z$.
However, in general it is neither $\vee$- nor $\wedge$-linear.

A finitely generated cone can be described as a max-plus column
span of a matrix $A\in\RR^{m\times n}$:
\begin{equation}
\label{span(A)}
\spann(A):=\{\mysup_{i=1}^n \lambda_i  + A_{\cdot i}\mid \lambda_i\in\RR,\ i=1,\ldots,n\}.
\end{equation}
In this case we denote $P_A:=P_{\spann(A)}$, and there is an explicit
expression for this operator which we recall below.

We denote $\RRbar:=\RR\cup\{+\infty\}$
and view $A\in\RRbar^{m\times n}$ as an operator from $\RRbar^m$ to $\RRbar^n$.
The {\em residuated operator} $A^{\diez}$ from $\RRbar^n$ to $\RRbar^m$ is
defined by
\begin{equation}
\label{adiez}
(A^{\diez} y)_j=y\mydiv A_{\cdot j}=\myinf_{i=1}^m (-a_{ij}+y_i),
\end{equation}
with the convention $(-\infty)+(+\infty)=+\infty$. Note that this operator, also known as
{\em Cuninghame-Green inverse}, sends
$\RRn$ to $\RRm$ whenever $A$ does not have columns equal to $-\infty$. The
term ``residuated'' refers to the property
\begin{equation}
\label{res-prop}
Ax\leq y\Leftrightarrow x\leq A^{\diez} y,
\end{equation}
where $\leq$ is the partial order on $\RRm$ or $\RRn$. Using \eqref{projector}
we obtain
\begin{equation}
\label{CG-proj}
P_A(z)=\mysup_{i=1}^n (z\mydiv A_{\cdot i})  + A_{\cdot i}=
AA^{\diez}z.
\end{equation}
In this form \eqref{CG-proj}, the nonlinear
projectors were studied by
Cuninghame-Green \cite{CG:79} (as $AA^*$-products).

Finitely generated cones are closed in the topology induced by the
metric
\begin{equation}
\label{dist-exp}
d(x,y)=\max\limits_i |e^{x_i}-e^{y_i}|,
\end{equation}
which coincides with Birkhoff's order topology. It is known
\cite[Theorem 3.11]{CGQS-05} that the
projectors onto such cones are continuous.

The intersection of two finitely generated cones can be
expressed in terms
of two-sided max-plus
linear systems with separated variables
$Ax=By$,
by the following proposition.

\begin{proposition}
\label{p:2sided-sep}
Let $A\in\RR^{m\times n_1}$ and $B\in\RR^{m\times n_2}$.
\begin{itemize}
\item [1.] If $(x,y)$ satisfies $Ax=By\neq\bzero$ then $z=Ax=By$
belongs to $\spann(A)\cap\spann(B)$. Equivalently,
$P_AP_B z=P_BP_A z=z$.
\item[2.] If $P_AP_B z=z\neq\bzero$ then there exist $x$ and $y$
such that $Ax=By=z$.
\end{itemize}
\end{proposition}

This approach to two-sided systems is also useful in the
case of systems with non-separated variables
$Ax=Bx$,
which is of greater importance for us here. This system
is equivalent to
\begin{equation}
\label{CDdef}
\begin{split}
Cx &= Dy,\ \text{where}\\
C=
\begin{pmatrix}
A\\
B
\end{pmatrix}, &\quad
D=
\begin{pmatrix}
I_m\\
I_m
\end{pmatrix},
\end{split}
\end{equation}
and $I_m=(\delta_{ij})\in\RR^{m\times m}$ denotes the max-plus
$m\times m$ identity matrix
with entries
\begin{equation}
\delta_{ij}=
\begin{cases}
0, & \text{if $i=j$},\\
\bzero, & \text{if $i\ne j$}.
\end{cases}
\end{equation}
In this case we have the following version of Proposition \ref{p:2sided-sep}.
\begin{proposition}
\label{p:2sided-nonsep}
Let $A,B\in\RR^{m\times n}$.
\begin{itemize}
\item [1.] If $x$ satisfies $Ax=Bx\neq\bzero$,
then $v=(z\; z)^T$, where $z=Ax=Bx$,
belongs to $\spann(C)\cap\spann(D)$. Equivalently,
$P_CP_D v=P_DP_C v=P_C v=v$.
\item[2.] If $v=(z\; z)^T\neq\bzero$ and $P_C v=v$, then there exist $x$
such that $Ax=Bx=v$.
\end{itemize}
\end{proposition}

Pairs $(x,y)\neq\bzero$ such
that $Ax=By=\bzero$ are described by:
$x_i\neq\bzero\Leftrightarrow A_{\cdot i}=\bzero$ and
$y_j\neq\bzero\Leftrightarrow B_{\cdot j}=\bzero$.
Analogously, vectors $x\neq\bzero$ such that
$Ax=Bx=\bzero$ are described by $x_i\neq\bzero\Leftrightarrow
A_{\cdot i}=B_{\cdot i}=\bzero$. Any such pair of vectors
can be added to any other pair $(x',y')$ or, respectively, vector $x'$,
and the resulting pair of vectors will
satisfy the system
if and only if so does $(x',y')$ or, respectively, $x'$.
Therefore, we can assume in the sequel without
loss of generality that there are no such solutions, i.e.,
that 1) $A$ and $B$ do not have $\bzero$ columns in the case of
separated variables, 2) $A$ and $B$ do not have common
$\bzero$ columns in the case of non-separated variables.

\subsection{Projectors and Perron-Frobenius theory}
Suppose that a function $f:\RRn\to\RRn$ is
homogeneous, isotone
and continuous in the topology
induced by \eqref{dist-exp}.
As $x\mapsto\exp(x)$ yields a
homeomorphism with $\Rp^n$ endowed with
the usual Euclidean topology, we can use
the spectral theory for homogeneous,
isotone and continuous functions in $\Rp^n$. We will use the following
important identities, which follow from the results of Nussbaum~\cite{Nus},
see~\cite[Lemma 2.8]{AGG-10} for the proof.
\begin{theorem}[Coro.~of~\cite{Nus},{\cite[Lemma 2.8]{AGG-10}}]
\label{Nussbaum}
Let $f$ denote an order-preserving, additively homogeneous
and continuous map from $(\R\cup\{-\infty\})^n$ to itself. Then
it has a largest eigenvalue
\[\rad(f)
:=\max\{\lambda\mid\exists x\in\RRn,\;x\not\equiv-\infty,\ \lambda+x= f(x)\},\\
\]
which coincides with 
\begin{gather}
\label{CW1}
\rad(f)=\max\{\lambda\mid\exists x\in\RRn,\;x\not\equiv-\infty,\ \lambda+x\leq f(x)\},\\
\label{CW2}
\rad(f)=\inf\{\lambda\mid\exists x\in\R^n,\ \lambda+x\geq f(x)\}.
\end{gather}
\end{theorem}
Note that \eqref{CW2} is nonlinear generalization of the classical Collatz-Wielandt
formula \cite{Minc}.
Equations \eqref{CW1} and \eqref{CW2}
are useful in max-plus algebra,
since they work for max-plus matrix multiplication as well as for
compositions of nonlinear projectors.
For \eqref{CW2} it is essential that
it is taken over vectors with real entries, and that the infimum
may not be reached.
Using \eqref{CW2} we obtain
that the spectral radius of such functions is isotone:
$f(x)\leq g(x)$ for all $x\in\R^n$ implies $\rad(f)\leq\rad(g)$.
We next recall an application of \eqref{CW2} to the
metric properties of compositions of
projectors, which appeared in~\cite{GS-08}. The {\em Hilbert distance}
between $u,v\in\RRn$ such that $\supp(u)=\supp(v)$ is defined by
\begin{equation}
\label{e:hilb}
\hilbdist(u,v)=\max\limits_{i,j\in\supp(v)}
(u_i-v_i+v_j-u_j).
\end{equation}
If $\spann(u)\neq\spann(v)$ then we set $\hilbdist(u,v)=+\infty$.
Using \eqref{e:hilb} we define the Hilbert distance between
cones $\spann(A)$ and $\spann(B)$, for $A\in\RR^{m\times n_1}$
and $B\in\RR^{m\times n_2}$:
\begin{equation}
\label{e:hilbmat}
\hilbdist(A,B):=\min\{\hilbdist(u,v)\mid u\in\spann(A),\; v\in\spann(B),\;
\supp(u)=\supp(v)\}.
\end{equation}
\begin{theorem}[cp. \cite{GS-08}, Theorem 25]
\label{t:rad-hilb}
Let $A\in\RR^{m\times n_1}$ and $B\in\RR^{m\times n_2}$.
Then
\begin{equation}
\label{e:rad-hilb1}
\rad(P_AP_B)=\rad(P_BP_A)=-\hilbdist(A,B).
\end{equation}
If $\hilbdist(A,B)$ is finite then
it is attained by any eigenvector $\overline{u}$ of $P_AP_B$
with eigenvalue $\rad(P_AP_B)$,
and its image $\overline{v}$ by $P_B$.
\end{theorem}
\begin{proof}
As $\supp(P_AP_Bu)\subseteq\supp(P_Bu)\subseteq\supp(u)$,
it follows that $P_AP_B$ and also $P_BP_A$ may have finite
eigenvalue only if $\spann(A)$ and $\spann(B)$ have vectors
with common support. This shows the claim for the case
$\hilbdist(A,B)=+\infty$.

Now let $\hilbdist(A,B)$ be finite.
We show that $-\hilbdist(\overline{u},\overline{v})=
-\hilbdist(A,B)=\rad(P_AP_B)$.
Take arbitrary vectors $u\in\spann(A)$ and $v\in\spann(B)$
with $\supp(u)=\supp(v)$, and
let $P_u$, resp.\ $P_v$, be projectors onto
the rays $U=\{\lambda+u,\,\lambda\in\RR\}$,
resp.\ $V=\{\lambda+v,\,\lambda\in\RR\}$.
As $U\subseteq\spann(A)$ and $V\subseteq\spann(B)$, we have
that $P_u\leq P_A$ and $P_v\leq P_B$,
hence $P_uP_v\leq P_AP_B$ and,
by the monotonicity of the spectral radius,
$\rad(P_uP_v)\leq \rad(P_AP_B)$.
It can be shown that $-\hilbdist(u,v)$ is the only finite eigenvalue of
$P_uP_v$, hence $-\hilbdist(u,v)=\rad(P_uP_v)$,
and consequently $-\hilbdist(u,v)\leq \rad(P_AP_B)$ and
$-\hilbdist(A,B)\leq \rad(P_AP_B)$.
Now observe that $-\hilbdist(\overline{u},\overline{v})=\rad(P_{\overline{u}}P_{\overline{v}})$ is
equal to the eigenvalue $\rad(P_AP_B)$. This completes the proof.\eproof
\end{proof}

In the case of the systems with non-separated variables, we will be
more interested in the Chebyshev  distance. For $u,v\in\RRm$ with
$\supp(u)=\supp(v)$ it is defined by
\begin{equation}
\label{e:chebdist}
\chebdist(u,v)=\max_{i\in\supp(v)} |u_i-v_i|.
\end{equation}

There is an important special case when Hilbert and Chebyshev
distances coincide.

\begin{lemma}
\label{hilb-cheb}
Let $u,v\in\RRm$ be such that $u\geq v$ and $u_i=v_i$ for some
$i\in\{1,\ldots,n\}$. Then $\hilbdist(u,v)=\chebdist(u,v)$.
\end{lemma}
\begin{proof} First note that both $\hilbdist(u,v)$ and $\chebdist(u,v)$ are finite if and only if
$\supp(u)=\supp(v)$. Only this case has to be considered.
 
If $u\geq v$ then $|u_j-v_j+v_l-u_l|\leq\max(u_j-v_j,u_l-v_l)$ for any $j$ and $l$, hence
$\hilbdist(u,v)\leq\chebdist(u,v)$. 

Fixing $l=i$ (assuming that $u_i=v_i$) we obtain
$u_j-v_j+v_l-u_l=u_j-v_j$ and fixing $j=i$ we obtain $u_j-v_j+v_l-u_l=v_l-u_l$. 
Taking maximum over such terms only yields $\chebdist(u,v)$, hence
$\chebdist(u,v)\leq\hilbdist(u,v)$.\eproof 
\end{proof}

\begin{theorem}
\label{t:rad-cheb}
Let $A,B\in\RR^{m\times n}$, and let $C$ and $D$ be defined as
in \eqref{CDdef}. Then
\begin{equation}
\label{e:rad-cheb}
\rad(P_CP_D)=\rad(P_DP_C)=-\min\limits_{x\in\RRm}
\chebdist(Ax, Bx).
\end{equation}
\end{theorem}

\begin{proof}
Theorem \ref{t:rad-hilb} implies that
\begin{equation}
\label{e:rad-hilb4}
\rad(P_CP_D)=-\min\{\hilbdist(u,v)\mid u\in\spann(C), v\in\spann(D).\}
\end{equation}

Let $u\in\spann(C)$ and denote by $P_u$ the projector onto
$U:=\{\lambda+u\mid \lambda\in\RR\}$.
Then $u$ is an eigenvector of $P_uP_D$ which corresponds to the spectral radius
of this operator, and applying Theorem~\ref{t:rad-hilb} to the max cones
$U$ and $\spann(D)$ we see that
\begin{equation}
\label{updudist}
\hilbdist(u,P_D u)=\min\{\hilbdist(u,v)\mid v\in\spann(D)\}.
\end{equation}
Note that \eqref{updudist} also holds if there is no $v\in\spann(D)$
with $\supp(u)=\supp(v)$, in which case $\hilbdist(u,P_D u)=+\infty$.
This implies
\begin{equation}
\label{e:rad-hilb3}
\rad(P_CP_D)=-\min\{\hilbdist(u,P_D u)\mid u\in\spann(C)\}.
\end{equation}
Observe that
\begin{equation}
u=
\begin{pmatrix}
Ax\\
Bx
\end{pmatrix}, \quad
P_D u=
\begin{pmatrix}
Ax\wedge Bx\\
Ax \wedge Bx
\end{pmatrix}
\end{equation}
for some $x\in\RRm$, and also that
$u$ and $P_D u$ satisfy the conditions of Lemma \ref{hilb-cheb}
unless $P_D u=\bzero$.
Hence $\hilbdist(u,P_D u)=\chebdist(u,P_D u)=\chebdist(Ax,Bx).$
Conversely, $\chebdist(Ax,Bx)$ equals $\hilbdist(u, P_D u)$ for
$u=(Ax\ Bx)^T$. Hence the r.h.s. of \eqref{e:rad-cheb} is the same as
the r.h.s. of \eqref{e:rad-hilb3}, which completes the proof.\eproof
\end{proof}

\subsection{Min-max functions and Chebyshev distance}
Let $A\in\RR^{m\times n_1}$ and $B\in\RR^{m\times n_2}$.
In order to find a point in the intersection of $\spann(A)$ and
$\spann(B)$ (or equivalently, solve $Ax=By$), one can compute
the action of $(P_AP_B)^l$, for $l=1,2,\ldots,$ on a vector
$z\in\RRm$. Dually one can start with a vector $x^0\in\RR^{n_1}$
and compute
\begin{equation}
\label{e:altmeth}
x^k=A^{\diez}BB^{\diez}Ax^{k-1},\quad k\geq 1.
\end{equation}
We can assume that $A$ and $B$ do not have columns equal to $\bzero$
so that $A^{\diez}z\in\RR^{n_1}$ and $B^{\diez} z\in\RR^{n_2}$
for any $z\in\RR^{m}$.

If at some stage $x^k=x^{k-1}\neq\bzero$ then we can stop,
$x^k$ is a solution of the system. If all coordinates of $x^k$ are less than those
of $x^0$ then we can stop, the system has no solution.
More details on this simple algorithm called {\em alternating method}
can be found in \cite{CGB-03} and \cite{Ser-09-inLS}, see
also~\cite{singer}.
In particular, it converges to a solution
with all components finite in a finite number of steps, if such a solution exists.

Let $A,B\in\RR^{m\times n}$. A system $Ax=Bx$ can be written equivalently as
$Cx=Dy$ with $C$ and $D$ as in \eqref{CDdef}. Applying
alternating method \eqref{e:altmeth} to this system, i.e., substituting
$C$ and $D$ for $A$ and $B$ in~\eqref{e:altmeth} we obtain $x^k=g(x^{k-1})$,
where
\begin{equation}
\label{e:gdef}
g(x)=A^{\diez}Ax\wedge B^{\diez}Bx\wedge
A^{\diez}Bx\wedge B^{\diez}Ax.
\end{equation}
As it is assumed that $A$ and $B$ do not have
common $\bzero$ columns and hence $C$ (and $D$) do not have
$\bzero$ columns, $g(x)\in\RRn$ for all $x\in\RRn$.

It can be shown that (see also \cite{CGB-03})
\begin{equation}
\label{rg0axbx}
\rad(g)=0\Leftrightarrow Ax=Bx\ \text{is solvable}.
\end{equation}
In particular, if $x$ is a fixed point of $g$ then it satisfies
$Ax=Bx$. For the function
\begin{equation}
\label{e:fdef}
f(x)=x\wedge A^{\diez}Bx\wedge B^{\diez}Ax
\end{equation}
which appears in \cite{DG-06},
it is also true the other way around, since
\begin{equation}
\label{fexpl}
\begin{split}
Ax=Bx &\Leftrightarrow Ax\geq Bx\ \&\
Bx\geq Ax\Leftrightarrow\\
&\Leftrightarrow B^{\diez}Ax\geq x\ \&\ A^{\diez}Bx\geq x\Leftrightarrow\\
&\Leftrightarrow x\wedge A^{\diez}Bx\wedge B^{\diez}Ax=x.
\end{split}
\end{equation}
We also introduce the function $h$:
\begin{equation}
\label{e:hdef}
h(x):=A^{\diez}Bx\wedge B^{\diez}Ax.
\end{equation}

Although $f$, $g$ and $h$ are different functions, they have
the same spectral radius, equal to the inverse minimal Chebyshev distance between
$Ax$ and $Bx$. To show this, we use the following identity.
\begin{equation}
\label{e:chebdist1}
-\chebdist(u,v)=\max\{\lambda\colon\lambda+u\leq v\ \&\ \lambda+v\leq u\}.
\end{equation}

\begin{theorem}
\label{rf=rg}
Let $A,B\in\RR^{m\times n}$.
For $C,D$ defined by \eqref{CDdef},
and $f$, $g$ and $h$ defined by \eqref{e:fdef}, \eqref{e:gdef} and \eqref{e:hdef},
\begin{equation}
\label{e:rf=rg}
\rad(P_CP_D)=\rad(P_DP_C)=\rad(f)=\rad(g)=\rad(h)=
-\min_{x\in\RRm} d_{\infty}(Ax,Bx).
\end{equation}
\end{theorem}

\begin{proof}
If $v$ is an eigenvector of $P_DP_C$ with a finite eigenvalue,
then $C^{\sharp}v$ is an eigenvector of $g$
and $P_C v$ is an eigenvector $P_CP_D$, both with the
same eigenvalue. The other way around, if
$x$ is an eigenvector of $g$ with a finite eigenvalue,
then $(Ax\ Bx)^T$ is an eigenvector of $P_DP_C$ with the same
eigenvalue.
This argument shows that 1) either the spectral radii
of $P_DP_C$, $P_CP_D$ and $g$ are all finite or they
all equal $\bzero$, 2) the equality
$\rad(g)=\rad(P_DP_C)=\rad(P_CP_D)$ holds true
both in finite and in infinite case.

We show the remaining equalities. By \eqref{CW1}, $\rad(h)$ is the maximum of $\lambda$
which satisfy
\begin{equation}
\exists x\in\RRn\colon \lambda+x\leq A^{\diez} Bx\wedge B^{\diez}Ax.
\end{equation}
This is equivalent to
\begin{equation}
\label{curious}
\exists x\in\RRn\colon \lambda+Ax\leq Bx \quad\&\quad \lambda+Bx\leq Ax
\end{equation}
Using \eqref{e:chebdist1} we obtain
\begin{equation}
r(h)=\max_{x\in\RRn} -\chebdist(Ax,Bx)=-\min_{x\in\RRn}\chebdist(Ax,Bx).
\end{equation}
It follows in particular that $r(h)\leq 0$ and moreover, $\lambda\leq 0$ for any $x$
satisfying \eqref{curious}. Applying \eqref{CW1} to $f$ and $g$ we obtain
that both $r(f)$ and $r(g)$ are equal to the maximum of $\lambda$ which satisfy
\begin{equation}
\exists x\in\RRm\colon \quad \lambda\leq 0 \quad \& \quad \lambda+Ax\leq Bx \quad\&\quad \lambda+Bx\leq Ax
\end{equation}
As the first inequality follows from the other two, we obtain $r(f)=r(g)=r(h)$.\eproof
\end{proof}

Functions $f$, $g$ and $h$ as well as projectors onto finitely
generated max-plus cones and their compositions, belong to the class
of {\em min-max functions}.
Such functions were originally considered by Olsder \cite{Ols-91} and
Gunawardena \cite{Gun-94}. See \cite{CGG-99} for a formal
definition. In a nutshell, these are additively homogeneous and
order preserving maps, every coordinate of which can be represented
as a minimum of a finite number of max-plus linear forms, or as a
maximum of a finite number of min-plus linear forms.  It is
important that any min-max function $q:\RRn\to\RRn$ can be
represented as infimum of finite number of max-plus linear maps
$Q^{(p)}$ meaning that
\begin{equation}
\label{e:maxrep}
q(x)=\myinf_p Q^{(p)}x,\\
\end{equation}
in such a way that the following {\em selection property}
is satisfied:
\begin{equation}
\label{e:select}
\forall x\ \exists p:\ q(x)=Q^{(p)} x.
\end{equation}
Note that taking infimum or supremum of vectors does not necessarily select one of them,
and that selection property~\eqref{e:select} is useful, e.g., for the
policy iteration algorithm of~\cite{DG-06}).  

In connection with the mean payoff games \cite{DG-06,AGG-10}, each matrix $Q^{(p)}$ corresponds to a one player
game, where the player Min has chosen her strategy and the player Max is trying to win what he can.

In particular, $f(x),g(x)$ and $h(x)$, respectively, are represented
as infima of the max-plus linear maps $F^{(p)},$ $G^{(p)}$ and
$H^{(p)}$, whose rows are taken from the max-plus linear forms
appearing in \eqref{e:fdef}, \eqref{e:gdef} and \eqref{e:hdef},
respectively, in the following way:
\begin{equation}
\label{e:fgmats}
F_{i\cdot}^{(p)}=
\begin{cases}
I_{i\cdot},\\
-a_{ki}  + B_{k\cdot},\\
-b_{ki}  + A_{k\cdot}.
\end{cases}\quad
G_{i\cdot}^{(p)}=
\begin{cases}
-a_{ki}  + A_{k\cdot},\\
-b_{ki}  + B_{k\cdot},\\
-a_{ki}  + B_{k\cdot},\\
-b_{ki}  + A_{k\cdot}.
\end{cases}\quad
H_{i\cdot}^{(p)}=
\begin{cases}
-a_{ki}  + B_{k\cdot},\\
-b_{ki}  + A_{k\cdot}.
\end{cases}
\end{equation}
Here $I_{i\cdot}$ denotes the $i$th row of the max-plus identity
matrix, and the brackets mean that any possibility, for any
$k=1,\ldots,m$ and $a_{ki}\neq\bzero$ or $b_{ki}\neq \bzero$, can be
taken (assumed that $A$ and $B$ do not have common $\bzero$
columns). Applying Collatz-Wielandt formula \eqref{CW2} we obtain
the following, some variants of which appeared in several contexts.

\begin{proposition}[Compare with~\cite{CGG-99,GG-98,gg0,AGK-10}]
\label{p:specrad-rep}
Suppose that a min-max function $q:\RRn\to\RRn$ is
represented as infimum of max-plus linear maps
$Q^{(l)}\in\RR^{n\times n}$ so that the selection property
is satisfied.
Then
\begin{equation}
\label{e:specrad-rep}
\rad(q)=\min\limits_l \rad(Q^{(l)}).
\end{equation}
\end{proposition}
\begin{proof}
The spectral radius is isotone,
hence $\rad(q)\leq\rad(Q^{(l)})$
for all $l$.
Using \eqref{CW2} we conclude that for any $\epsilon$ there is $x\in\R^m$ such that
$q(x)\leq \rad(q)+\epsilon+x$. As $q(x)=Q^{(l)}x$
for some $l$ and there is only finite number of matrices
$Q^{(l)}$, there exists
$l$ such that
\begin{equation}
\label{CW-lowenv}
\rad(q)=\inf\{\mu\mid\exists x\in\R^n,\ Q^{(l)}x\leq\mu  + x\}=
\rad(Q^{(l)}).
\end{equation}
The proof is complete.\eproof
\end{proof}
Proposition~\ref{p:specrad-rep} can be derived alternatively from
the duality theorem in~\cite[Theorem 19]{GG-98} (see also~\cite{gg0}).
It is related to the existence of the value of stochastic games with perfect information~\cite{LL-69}. Indeed, the spectral radius can be
seen to coincide with the value of a game in which Player Max
chooses the initial state, see~\cite{AGG-10} for more information.

The greatest eigenvalue $\rad(Q^{(l)})$ of the max-plus matrix
$Q^{(l)}=(q^{(l)}_{ij})\in\RR^{n\times n}$ can be computed explicitly.
It is equal to the maximum cycle mean
of $Q^{(l)}$ defined by
\begin{equation}
\label{mcmh}
\max\limits_{1\leq k\leq n} \max\limits_{i_1,\ldots,i_k}
\frac{q^{(l)}_{i_1i_2}+q^{(l)}_{i_2i_3}+\ldots +q^{(l)}_{i_ki_1}}{k}.
\end{equation}
This result is fundamental in max-plus algebra, see~\cite{abg05,BCOQ,But:10,HOW:05}
for more details.

\section{The spectrum and the spectral function}

\subsection{Construction of the spectral function}

Given $A\in\RRmm$ and $B\in\RRmm$,
we consider the {\em two-sided eigenproblem}
which consists in finding {\em eigenvalues}
$\lambda\in\RR$ and {\em eigenvectors}
$x\in\RRn$ (which have at least one component not equal to $\bzero$),
such that
\begin{equation}
\label{2sided}
A x=\lambda  + Bx.
\end{equation}
The set of eigenvalues is called the
{\em spectrum of} $(A,B)$
and denoted by $\spec(A,B)$. 

Below we assume that $A$ and $B$ do not have $-\infty$ rows
and common $-\infty$ columns.  Note that the assumption about $-\infty$ rows can be
made without loss of generality when the solvability of~\eqref{2sided} is considered.  
Indeed, if the $i$th row of $B$ is $-\infty$ then all variables
$x_j$ such that $a_{ij}\neq -\infty$ must be equal to $-\infty$. Eliminating these 
variables as well as the corresponding columns in $A$ and $B$ and the $i$th equation, we obtain a new system where $A$ or $B$ may have $-\infty$ rows. Proceeding this way we either cancel the whole system
in which case it is unsolvable, or we are left with a system where $A$ and $B$ (what remains of them) do not have $-\infty$ rows. This procedure can be run in $O(m^2n)$ operations.

The case of $\lambda=\bzero$ appears if and only if
$A$ has $\bzero$ columns, and the corresponding
eigenvectors are described by
$x_i\neq\bzero\Leftrightarrow A_{\cdot i}=\bzero$.
In the sequel we assume that $\lambda$ is finite.

Problem \eqref{2sided} is equivalent to $C(\lambda)x=Dy$,
where $C(\lambda)\in\RR^{2m\times n}$ and $D\in\RR^{2m\times m}$ are defined by
\begin{equation}
\label{ClD}
C(\lambda)=
\begin{pmatrix}
A\\
\lambda  + B
\end{pmatrix},\quad
D=
\begin{pmatrix}
I_m\\
I_m
\end{pmatrix}.
\end{equation}
As it follows from Theorem~\ref{rf=rg}, $\spec(A,B)=\{\lambda\colon
r(P_DP_{C(\lambda)})=0\}=\{\lambda\colon r(h_{\lambda})=0\}$, where
\begin{equation}
\label{flambda}
h_{\lambda}(x)=(\lambda+A^{\diez}Bx)\wedge(-\lambda+B^{\diez}Ax).
\end{equation}
The function $h_{\lambda}$ can be represented as infimum of max-plus
linear maps so that the selection property \eqref{e:select} is
satisfied. Namely,
\begin{equation}
\label{flambda2}
h_{\lambda}(x)=\myinf_p H_{\lambda}^{(p)}x,
\end{equation}
where for $i=1,\ldots,n$
\begin{equation}
\label{flambdarows}
(H_{\lambda}^{(p)})_{i\cdot}=
\begin{cases}
\lambda-a_{ki} +B_{k\cdot}, & \text{for $1\leq k\leq n$, $a_{ki}\neq\bzero$},\\
-\lambda -b_{ki}  + A_{k\cdot}, & \text{for $1\leq k\leq n$, $b_{ki}\neq\bzero$},
\end{cases}
\end{equation}
the brackets meaning that any listed choice can be taken.

The greatest eigenvalue of $H_{\lambda}$ equals the maximum cycle
mean of $H_{\lambda}$. Using formula \eqref{mcmh}, we observe that
$\rad(H_{\lambda})$ is a piecewise-affine function, meaning that it
is composed of a finite number of affine pieces. More precisely, we
have the following.

\begin{proposition}
\label{p:flambdaconv} 
$\rad(H_{\lambda}^{(p)})$ is a finite
piecewise-affine convex Lipschitz function of $\lambda$.
\end{proposition}
\begin{proof}
Using \eqref{mcmh} we observe that $\rad(H_{\lambda}^{(p)})=\bzero$
if and only if the associated digraph of $H_{\lambda}^{(p)}$ is
acyclic, which cannot happen when $A$ and $B$ and hence $H_{\lambda}^{(p)}$
do not have $-\infty$ rows.

If $\rad(H_{\lambda}^{(p)})$ is finite, then any finite cycle mean
of $H_{\lambda}^{(p)}$ can be written as $(k\lambda+a)/l$, where $l$
is the length of the cycle and $k$ is an integer number with modulus
not greater than $l$, hence this affine function is Lipschitz. The
function $\rad(H_{\lambda}^{(p)})$ is pointwise maximum of a finite
number of such affine functions, hence it is a convex Lipschitz
piecewise-affine function.\eproof
\end{proof}

\begin{definition}[Spectral Function]
We define the {\em spectral function} of~\eqref{2sided} by
\begin{equation}
\label{specfunc-def}
s(\lambda):=\rad(h_{\lambda})=\rad(P_DP_{C(\lambda)}).
\end{equation}
\end{definition}

It follows from Theorem~\ref{rf=rg} that $s(\lambda)\leq 0$ and that
$s(\lambda)=0$ if and only if $\lambda\in\spec(A,B)$. In general,
$s(\lambda)$ is equal to the inverse minimal Chebyshev distance
between $Ax$ and $\lambda+Bx$.

By Proposition \ref{p:specrad-rep},
\begin{equation}
\label{e:low-env}
s(\lambda)=\myinf_p\rad(H_{\lambda}^{(p)}).
\end{equation}

As $\rad(H_{\lambda}^{(p)})$ are piecewise-affine and Lipschitz, we
conclude the following.

\begin{corollary}
$s(\lambda)$ is a finite piecewise-affine Lipschitz function.
\end{corollary}

\if{
Let us consider the case $s(\lambda)=-\infty$ in more detail. 
For arbitrary $C=(c_{ij})\in\Rmax^{m\times n}$ define
$C^{\circ}=(c_{ij}^{\circ})\in\Rmax^{m\times n}$ by
\begin{equation}
\label{abcirc}
c_{ij}^{\circ}=
\begin{cases}
0, & \text{if $c_{ij}\in\R$},\\
-\infty, & \text{if $c_{ij}=-\infty$}.
\end{cases}.
\if{
b_{ij}^{\circ}=
\begin{cases}
0, & \text{if $b_{ij}\in\R$},\\
-\infty, & \text{if $b_{ij}=-\infty$}.
\end{cases},\quad
}\fi
\end{equation}
The spectral function of the eigenproblem
$A^{\circ}x=\lambda+B^{\circ}x$ will be denoted by
$s^{\circ}(\lambda)$.

\begin{proposition}
\label{1stboolean}
The following are equivalent:
\begin{itemize}
\item[1.] $s(\lambda)$ is finite for all $\lambda$;
\item[2.] $s^{\circ}(\lambda)$ is finite for all $\lambda$;
\item[3.] $A^{\circ}x=B^{\circ}x$ has a nontrivial solution whose
entries belong to $\{0,-\infty\}$.
\end{itemize}
\end{proposition}
\begin{proof}

$1.\Leftrightarrow 2:$ $s(\lambda)=-\infty$ if and only if there
exists $H_{\lambda}^{(p)}$ such that
$\rad(H_{\lambda}^{(p)}=-\infty$. By \eqref{mcmh}, this just means
that the associated digraph of $H_{\lambda}^{(p)}$ does not have
cycles with finite weight. This property does not depend on the
value of finite coefficients in $H_{\lambda}^{(p)}$ and hence
$s(\lambda)=-\infty$ if and only if $s^{\circ}(\lambda)=-\infty$.

\if{$3.\Leftrightarrow 4:$ We need only to show that if
$A^{\circ}x=B^{\circ}x$ has a solution, then it has a solution with
$\{0,-\infty\}$ entries. Let $x$ be a solution and denote
$t_k:=(A^{\circ}x)_k=(B^{\circ}x)_k$. If $t_k=-\infty$, then the
equality cannot be affected by any change of finite entries of $x$.
If $t_k$ is finite, then there exist indices $l$ and $m$ such that
$a_{kl}=b_{km}=0$, and setting all entries of $x$ to $0$ we obtain
$(A^{\circ}x)_k=(B^{\circ}x)_k=0$.}\fi

$3.\Leftrightarrow 1:$ As $s(\lambda)$ is equal to the inverse
minimal Chebyshev distance between $Ax$ and $Bx$, it is infinite if
and only if there is no nontrivial vector $x$ such that
$\supp(Ax)=\supp(Bx)$, which is the negation of 3.\eproof
\end{proof}

Condition 3. of Proposition~\ref{1stboolean} provides a criterion
for $s(\lambda)=-\infty$, which can be verified in polynomial time.
}
\fi

Let us indicate yet another consequence of the fact that
$r(H_{\lambda}^{(p)})$ and $s(\lambda)$ are piecewise-affine.

\begin{corollary}
If $\spec(A,B)$ is not empty, then it
is a finite system of closed intervals and points.
\end{corollary}

Note that this also follows, by means of projection, from
a result by De Schutter and De Moor~\cite{SM-96} that the solution
set of a system of polynomial (in)equalities in the max-plus algebra
is a (finite) union of polyhedra. The method of De Schutter and De Moor can also offer an alternative (computationally expensive) way to determine the spectrum and the generalized eigenvectors.

Conversely, it is shown in \cite{Ser-note-10} that any system of closed intervals
and points in $\R$ can be represented as spectrum of $(A,B)$. See also Subsect.~\ref{ss:example}.

\subsection{Bounds on the spectrum of $(A,B)$}
\label{ss:bounds}

Next we recall a bound
on the spectrum obtained
by Cuninghame-Green and Butkovi\v{c} \cite{But:10,CGB-08},
extending it to the case when $A=(a_{ij})$ and
$B=(b_{ij})$ may have infinite entries.
Denote
\begin{equation}
\label{but-bounds}
\begin{split}
&\underline{D}(A,B)=
\mysup_{i\colon A_{i\cdot}\; \text{finite}} A_{i\cdot}\mydiv B_{i\cdot},\\
&\overline{D}(A,B)=
-\mysup_{i\colon B_{i\cdot}\; \text{finite}} B_{i\cdot}\mydiv A_{i\cdot}.
\end{split}
\end{equation}
We assume that $\mysup\emptyset=-\infty$ and $-\mysup\emptyset=+\infty$.

Since $A_{i\cdot}\mydiv B_{i\cdot}=
\max\{\gamma\mid A_{i\cdot}\geq \gamma+B_{i\cdot}\}$
is finite when $A_{i\cdot}$ is finite and $B_{i\cdot}$ is not
$\bzero$, we immediately see the following.

\begin{lemma}
\label{but-finite}
$\underline{D}(A,B)$ (resp.\ $\overline{D}(A,B)$)
is finite if and only if there exists an $i\in\{1,\ldots,m\}$
such that $A_{i\cdot}$ is finite (resp.\ $B_{i\cdot}$
is finite).
\end{lemma}

When $A$ and $B$ have finite entries only,
$\underline{D}(A,B)$ and $\overline{D}(A,B)$ are just like
the bounds of \cite[Theorem 2.1]{CGB-08}:
\begin{equation}
\label{but-orig-bounds}
\begin{split}
&\underline{D}(A,B)=
\mysup_i\myinf_j (a_{ij}-b_{ij}),\\
&\overline{D}(A,B)=\myinf_i \mysup_j (a_{ij}-b_{ij}).
\end{split}
\end{equation}
Note that $\underline{D}(A,B)$ and $\overline{D}(A,B)$ defined by \eqref{but-bounds} take infinite values
if  $A$ or $B$ do not contain any finite rows.

\begin{proposition}
\label{peter}
If $Ax\leq\lambda+Bx$ (resp.\ $Ax\geq\lambda+Bx$) has solution
$x>\bzero$, then $\lambda\geq\underline{D}(A,B)$
(resp.\ $\lambda\leq\overline{D}(A,B)$).
\end{proposition}
\begin{proof}
If there exists $i$ such that $a_{ij}>\lambda+b_{ij}$ for all
$j=1,\ldots,m$, then $Ax\leq\lambda+Bx$ cannot have solutions.
This condition is equivalent to $A_{i\cdot}\mydiv B_{i\cdot}>\lambda$
plus the finiteness of $A_{i\cdot}$. Taking supremum
of $A_{i\cdot}\mydiv B_{i\cdot}$ over $i$ such that $A_{i\cdot}$ is
finite yields $\underline{D}(A,B)$. This shows that
if $Ax\leq\lambda+Bx$
 then $\lambda\geq\underline{D}(A,B)$.
The remaining part follows analogously.\eproof
\end{proof}

The next result is an extension of \cite[Theorem 2.1]{CGB-08}.

\begin{corollary}
\label{c:butbounds}
$\spec(A,B)\subseteq
[\underline{D}(A,B),\overline{D}(A,B)]$.
\end{corollary}

We use identity
\eqref{CW1} to give a more precise bound. It will be assumed
that $A$ and $B$ do not have $\bzero$ columns. Note that
this condition is more restrictive than that $A$ and $B$ do not
have {\em common} $\bzero$ columns, and it cannot be assumed
without loss of generality.

\begin{theorem}
\label{p:specbounds1}
Suppose that $A=(a_{ij}),B=(b_{ij})\in\RR^{m\times n}$
do not have $\bzero$ columns. Then
\begin{equation}
\label{e:specbounds1}
\spec(A,B)\subseteq [-\rad(A^{\diez}B),\rad(B^{\diez}A)]
\subseteq [\underline{D}(A,B),\overline{D}(A,B)].
\end{equation}
\end{theorem}
\begin{proof}
Let $Ax=\lambda Bx$, then we also have
\begin{equation}
\label{equivs1}
\begin{split}
& Ax\leq\lambda+Bx\Leftrightarrow -\lambda+x\leq A^{\diez}Bx,\\
& \lambda+ Bx\leq Ax\Leftrightarrow \lambda+x\leq B^{\diez}Ax.
\end{split}
\end{equation}
As $A$ and $B$ do not have $\bzero$ columns so that
$A^{\diez}Bx$ and $B^{\diez}Ax$ do not have $+\infty$ entries,
we can use \eqref{CW1} to obtain from \eqref{equivs1} that
$\lambda\in [-\rad(A^{\diez}B),\rad(B^{\diez}A)]$. For
$\lambda=\rad(B^{\diez}A)$ we can find $y\neq\bzero$ such that
$\lambda+y\leq B^{\diez}Ay$ and hence $\lambda+By\leq Ay$.
Using Proposition \ref{peter} we obtain
$\lambda\leq\overline{D}(A,B)$. The remaining inequality
$\lambda\geq\underline{D}(A,B)$ can be obtained analogously.\eproof
\end{proof}
By comparison with the finer bounds $-\rad(A^\sharp B)$ and $\rad(B^\sharp A)$, the interest of the bounds of Butkovi\v{c} and Cuninghame-Green, $\underline{D}(A,B)$ and $\overline{D}(A,B)$, lies in their explicit character. However, these bounds become infinite when the matrices $A$ and $B$ do not have any finite rows. We next give different explicit bounds, which turn out to be finite as soon as $A$ and $B$ do not have any identically infinite columns.
\begin{proposition}
\label{specints}
We have
\[
\spec(A,B)\subseteq \bigcup_{1\leq i\leq n} [-(A^\sharp B0)_i, (B^\sharp A 0)_i] \enspace ,
\]
and so
\[
\spec(A,B)\subseteq [-\mysup_i (A^\sharp B0)_i, \mysup_i  (B^\sharp A 0)_i ],
\]
where $0$ is the $n$-vector of all $0$'s.
\end{proposition}
\begin{proof}
Consider $x:=0$ and $\mu:=\mysup_i [h_\lambda(0)]_i$, so that
$h_\lambda(x)\leq \mu +x$. Then, the non-linear Collatz-Wielandt
formula~\eqref{CW2} implies that $r(h_\lambda)\leq \mu$. If
$\lambda\in \spec(A,B)$, we have $0\leq r(h_\lambda)$, and so, there
exists at least one index $i\in\{1,\ldots,n\}$ such that
\[
0\leq [h_\lambda(0)]_i =(\lambda + (A^\sharp B 0)_i )\wedge
(-\lambda + (B^\sharp A0)_i) \enspace .
\]
It follows that $\lambda \leq (B^\sharp A0)_i$ and $\lambda\geq - (A^\sharp B 0)_i$.\eproof
\end{proof}
\begin{remark}
It follows readily from the Collatz-Wielandt property~\eqref{CW2} that
\[
[-r(A^\sharp B),r(B^\sharp A)]\subseteq [-\mysup_i (A^\sharp B0)_i, \mysup_i  (B^\sharp A 0)_i ]
\]
\typeout{Give example showing that Peter and Ray's bound and our CW
bound are uncomparable. Example given.}
\end{remark}

\begin{figure}
\centering
\includegraphics[width=0.8\linewidth]{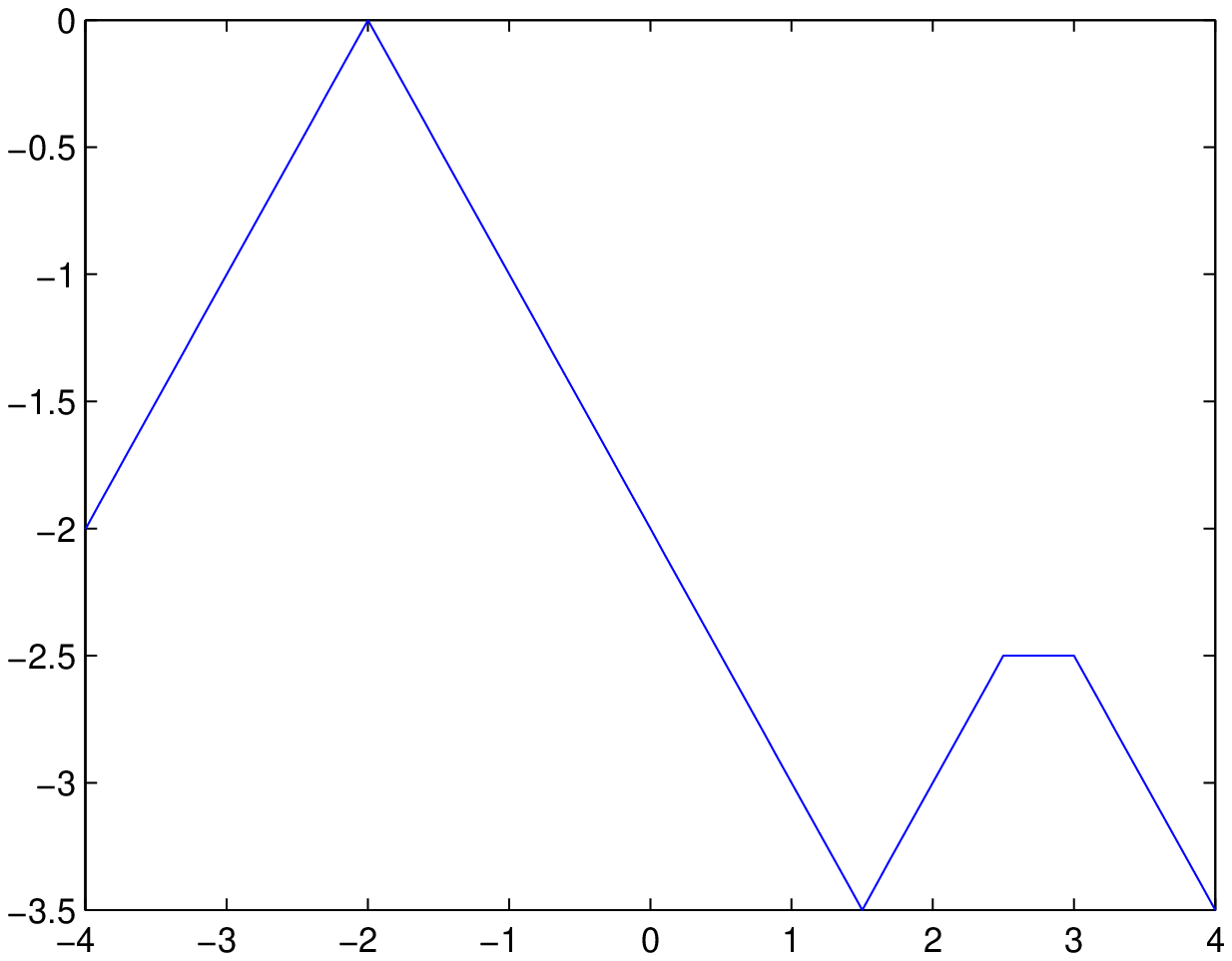}
\caption{Spectral function of \eqref{ABrand}}
\label{curioustest}
\end{figure}

\begin{example}
{\rm We next give an example, to compare the bounds of Corollary
\ref{c:butbounds}, Theorem~\ref{p:specbounds1} and
Proposition~\ref{specints}. Consider the following finite matrices
of dimension $3\times 4$:
\begin{equation}
\label{ABrand}
A =
\begin{pmatrix}
-2  &   3 &   -3  &  -3\\
-4  &   1  &   2  &  -2\\
5  &  -1   &  5   & -1
\end{pmatrix},\quad
B =
\begin{pmatrix}
-4   &  5 &   -3 &    3\\
 2  &   0  &  -1   &  4\\
 0   &  2  &  -3   & -1
\end{pmatrix}
\end{equation}
From the graph of spectral function, Figure~\ref{curioustest}, it
follows that the only eigenvalue is $-2$ since $s(-2)=0$ and $s(\lambda)<0$ for any
$\lambda\neq -2$. The interval
$[-r(A^{\diez}B),r(B^{\diez}A)]$ is in this case $[-2,0.5]$.  Bounds
\eqref{but-orig-bounds} of \cite[Theorem 2.1]{CGB-08} yield the
interval $[\underline{D}(A,B),\overline{D}(A,B)]=[-3,2]$, which is
less precise. Proposition~\ref{specints} yields the union of
intervals $[3,0]=\emptyset$, $[-2,-2]$, $[3,3]$ and $[-3,-2]$, thus
$[-3,-2]\cup \{3\}$. Note that these intervals are incomparable both
with $[-r(A^{\diez}B),r(B^{\diez}A)]$ and
$[\underline{D}(A,B),\overline{D}(A,B)]=[-3,2]$.}
\end{example}

We remark that the intervals $[-\mysup_i (A^\sharp B0)_i, \mysup_i
(B^\sharp A 0)_i ]$ and $[\underline{D}(A,B),\overline{D}(A,B)]$ are
also in general incomparable. Also, Subsect. \ref{ss:example} will
provide an example where the bounds $[-r(A^{\diez}B),r(B^{\diez}A)]$
are exact.
\begin{example}
Let us now illustrate the discrete event systems interpretation of the spectral
problem of the previous example. For readability, we replace the matrices
by 
\begin{equation}
\label{ABrand2}
A =
\begin{pmatrix}
-2  &   3 &   -\infty  &  -\infty\\
-\infty  &   1  &   2  &  -\infty\\
5  &  -\infty   &  5   & -1
\end{pmatrix},\quad
B =
\begin{pmatrix}
-\infty   &  5 &   -3 &    -\infty\\
 2  &   -\infty  &  -\infty   &  4\\
 0   &  2  &  -\infty   & -\infty
\end{pmatrix}
\end{equation}
This pair of matrices can be shown to have the same spectral function (Figure~\ref{curioustest}) as the previous one, and the same bounds $[-r(A^{\diez}B),r(B^{\diez}A)]$. Consider now the two discrete event
systems
\[
y=A x,\qquad z=Bx \enspace .
\]
Here, $x_i$ is interpreted as the starting time of a task $i$,
and $y_i$ and $z_i$ are interpreted as output time. This is illustrated
in Figure~\ref{fig-new}. For instance, the constraint $y_1=\max(-2+x_1,3+x_2)$
in $y=Ax$ expresses that the first output is released at the earliest,
given that it must wait $3$ time units after the second input becomes available,and can not be released more than $2$ time units before the first input becomes
available. We are looking for a common input $x$ such that the time separation
between events is the same for both outputs, so that
\[
y_i-y_j=z_i-z_j,\qquad \forall i,j \enspace .
\]
This can be solved by finding an eigenvector $x$, so that $Ax=\lambda+ Bx$. 
By inspection of the spectral function in Figure~\ref{curioustest}, we see
that $\lambda$ must be equal to $-2$. Then, computing $x$ reduces to solving a mean payoff game (see the discussion in section~\ref{ss:MPG} below for more algorithmic background). 
In this special example, $x$ can be determined very simply by running
the power type algorithm (like the alternating method of~\cite{CGB-03})
\[
x^{(0)}=(0,0,0,0)^T,\qquad x^{(k+1)}=h_{-2}(x^{(k)})
\]
where
\[
h_{-2}(x):=(-2+A^{\sharp}Bx)\enspace\wedge\enspace
(2+B^{\sharp}Ax) \enspace ,
\]
until the sequence $x^k$ converges. Actually,
\[
x^{(2)}=x^{(3)}=(-5,0,-5,-1)^T\enspace,
\]
and it can be checked that
\[
Ax^{(2)}=-2+Bx^{(2)} = (3,1,0)^T \enspace .
\]
\begin{figure}
\begin{tabular}{ccc}
\input{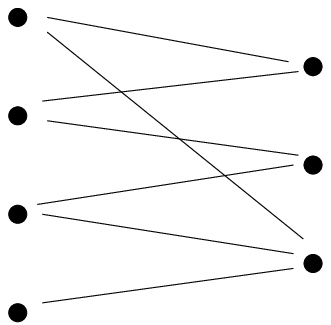}&\ \ \ \ \ \ \ \ \ \ \ \ \ \ \ \ \ &
\input{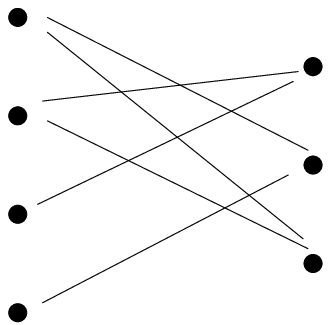}
\end{tabular}
\caption{Finding a common input making the outputs of two discrete event systems
indistinguishable, modulo a constant}
\end{figure}
\label{fig-new}
\end{example}

\subsection{Asymptotics of the spectral function}
If $A$ and $B$ do not have $\bzero$ columns,
the functions $\lambda+ A^{\diez}B$ and $-\lambda +B^{\diez}A$
are represented as infima of all max-linear mappings
$K_{\lambda}^{(p)}$ and, respectively, $M_{\lambda}^{(s)}$
such that
\begin{equation}
\label{KMdef}
\begin{split}
(K_{\lambda}^{(p)})_{i\cdot}&=
\lambda  - a_{ki}  + B_{k\cdot},
\quad 1\leq k\leq n,\ a_{ki}\neq\bzero,\\
(M_{\lambda}^{(s)})_{i\cdot}&=
-\lambda  - b_{ki}  + A_{k\cdot},
\quad 1\leq k\leq n,\ b_{ki}\neq\bzero.
\end{split}
\end{equation}
This representation satisfies the selection property.

Matrices $K_{\lambda}^{(p)}$ and $M_{\lambda}^{(s)}$ are both
instances of $H_{\lambda}^{(p)}$ which represent $h_{\lambda}$. 
We will need the following observation on $\rad(H_{\lambda}^{(p)})$

\begin{lemma}
\label{l:maxlength}
Denote $\maxlength:=\min(2m,n)$. The spectral radii $\rad(H_{\lambda}^{(p)})$ can be
expressed as $\lambda s/l+\alpha$, where $0\leq |s|\leq
l\leq \maxlength$, and $\alpha \leq\Delta(A,B)$, where
\begin{equation}
\label{Delta-def}
\Delta(A,B):=
\bigvee_{i,j,k\colon a_{ij}\neq\bzero,\; b_{ik}\neq\bzero} (a_{ij}  - b_{ik})
\vee
\bigvee_{i,j,k\colon b_{ij}\neq\bzero,\; a_{ik}\neq\bzero} (b_{ij}  - a_{ik}).
\end{equation}
Moreover it is only possible
that $s=l-2t$ for $t=0,\ldots,l$.  
\end{lemma}
\begin{proof}
According to~\eqref{mcmh} and~\eqref{flambdarows}, $\rad(H_{\lambda}^{(p)})$ is a cycle mean of the form
\begin{equation}
\label{hmean}
(t_{i_1i_2}^{k_1}+\ldots+t_{i_li_1}^{k_l})/l
\end{equation}
where we use the notation $t_{ij}^{k}:=-a_{ki}+\lambda+b_{kj}$ and 
$t_{ij}^{n+k}=-\lambda-b_{ki}+a_{kj}$ for $i,j=1,\ldots,n$ and $k=1,\ldots,m$. 
Actually $t_{ij}^{k}$ is the $(i,j)$ entry of $H_{\lambda}^{(p)}$, but here we also need the intermediate index $k$. Note that it is determined by $i$.
  
In~\eqref{hmean}, only even numbers of $\pm\lambda$ can be cancelled, hence it can be expressed as
$\lambda s/l+\alpha$ where $0\leq |s|\leq
l$ with $s=l-2t$ for $t=0,\ldots,l$. The cycle $(i_1,\ldots,i_l)$ is elementary, hence $l\leq n$.
We also obtain $\alpha \leq\Delta(A,B)$ since the arithmetic mean does not exceed maximum.

It remains to show that $l\leq 2m$. Indeed if $l>2m$ then there is an upper index which appears at least twice in~\eqref{hmean}. Assume w.l.o.g. that this is $k_1$. Then the sum in~\eqref{hmean} takes one of the
following forms:
\begin{equation}
\label{hmean1}
\begin{split}
&-a_{k_1i_1}+[\lambda+b_{k_1i_2}+\ldots-a_{k_1i_r}]+\lambda+b_{k_1i_{r+1}}+\ldots,\\
&-\lambda-b_{k_1i_1}+[a_{k_1i_2}+\ldots-\lambda-b_{k_1i_r}]+a_{k_1i_{r+1}}+\ldots
\end{split}
\end{equation}
Assume w.l.o.g. that we have the first one. Then we can split it into the following two cycles, as indicated by the square bracket in the first line of~\eqref{hmean1}:
\begin{equation}
\label{twocycles}
\begin{split}
&t_{i_ri_2}^{k_1}+t_{i_2i_3}^{k_2}\ldots+t_{i_{r-1}i_r}^{k_{r-1}},\\
&t_{i_1i_{r+1}}^{k_1}+t_{i_{r+1}i_{r+2}}^{k_{r+1}}\ldots+t_{i_li_1}^{k_l}.
\end{split}
\end{equation}
Each of these forms is a weight of a cycle in $H_{\lambda}^{(p)}$. Indeed, \eqref{hmean1} (the first expression)
indicates that $k_1$ is chosen by $i_r$ in $H_{\lambda}^{(p)}$ so that any element $t_{i_rj}^{k_1}$ for
$j=1,\ldots,n$ is an entry of $H^{(p)}$. All other elements in~\eqref{twocycles} are also entries of $H^{(p)}$. 

The arithmetic mean for both of the cycles in~\eqref{twocycles}
has to be equal to~\eqref{hmean}, if this is indeed $\rad(H_{\lambda}^{(p)})$.
This shows $l\leq 2m$.\eproof
\end{proof}

We also define
\begin{equation}
\label{Cdef}
\begin{split}
\overline{C}(A,B)&:=
\bigvee_{i,j,k\colon a_{ij}\neq\bzero,\; b_{ik}\neq\bzero} (a_{ij} - b_{ik}),\\
\underline{C}(A,B)&:=
\bigwedge_{i,j,k\colon a_{ik}\neq\bzero,\; b_{ij}\neq\bzero} (a_{ik}-b_{ij}).
\end{split}
\end{equation}

We now study
the asymptotics of $s(\lambda)$, both in general
case and in some special cases.

\begin{theorem}
\label{t:asymp}
Suppose that $A,B\in\RR^{m\times n}$ and denote $\maxlength:=\min(2m,n)$.
\begin{itemize}
\item[1.]  
There exist $k_1$, $l_1$, $k_2$, $l_2$ such that
$0\leq l_1\leq \maxlength$, $k_1=l_1-2t_1$ where $0\leq t_1\leq\lfloor
l_1/2\rfloor$, $0\leq l_2\leq \maxlength$, $k_2=l_2-2t_2$ where $0\leq
t_2\leq\lfloor l_2/2\rfloor$, and $\alpha_1,\alpha_2\in\R$ such
that
\begin{equation}
\label{asyunu}
\begin{split}
s(\lambda)&=\lambda k_1/l_1+\alpha_1,\quad
\text{if $\lambda\leq -2\maxlength^2\mytimes\Delta(A,B)$},\\
s(\lambda)&=-\lambda k_2/l_2+\alpha_2.\quad
\text{if $\lambda\geq 2\maxlength^2\mytimes\Delta(A,B)$}.
\end{split}
\end{equation}
\item[2.] Suppose that $A$ and $B$ do not have
$\bzero$ columns. 
Then there exist $\alpha_1\leq \rad(A^{\diez}B)$
and $\alpha_2\leq \rad(B^{\diez}A)$ such that
\begin{equation}
\label{asydoi}
\begin{split}
s(\lambda)&=\lambda+\alpha_1,\quad
\text{if $\lambda\leq -2\maxlength\mytimes\Delta(A,B)$},\\
s(\lambda)&=-\lambda+\alpha_2,\quad
\text{if $\lambda\geq 2\maxlength\mytimes\Delta(A,B)$}.
\end{split}
\end{equation}
\item[3.] Suppose that $A$ and $B$ are real.
Then
\begin{equation}
\label{asytrei}
\begin{split}
s(\lambda)&=\lambda+\rad(A^{\diez}B),\quad
\text{if $\lambda\leq\underline{C}(A,B)$},\\
s(\lambda)&=-\lambda+\rad(B^{\diez}A),\quad
\text{if $\lambda\geq\overline{C}(A,B)$}.
\end{split}
\end{equation}
\end{itemize}
\end{theorem}

\begin{proof}
1: For the proof of this part, we observe that for each $\lambda$,
the function $s(\lambda)$ is the maximum cycle mean of a
representing matrix $H_{\lambda}^{(p)}$, so that it equals $\lambda
k/l+\alpha$ where $0\leq l\leq \maxlength$, $k=l-2t$ where $0\leq t\leq l$.
For any two such terms, difference between coefficients $k/l$ is not
less than $1/\maxlength^2$, and the difference between the offsets does not
exceed $2\Delta(A,B)$, which yields that all intersection points
must be in the interval
$[-2\maxlength^2\mytimes\Delta(A,B),2\maxlength^2\mytimes\Delta(A,B)]$.  Thus $s(\lambda)$
is just one affine piece for $\lambda\leq -2\maxlength^2\mytimes\Delta(A,B)$ and for
$\lambda\geq 2\maxlength^2\mytimes\Delta(A,B)$. As $s(\lambda)\leq 0$ for all
$\lambda$, the left asymptotic slope is nonnegative, and the
right asymptotic slope is non-positive.\\
2: When $A$ does not have $\bzero$ columns, some of the matrices
$H_{\lambda}^{(p)}$ are of the form $K_{\lambda}^{(p)}$ and their
maximum cycle mean is $\lambda+\alpha$. Taking minimum over all
$\rad(H_{\lambda}^{(p)})$ of the form $\lambda+\alpha$ yields an offset
$\alpha_1\leq\rad(A^{\diez}B)$. The cycle mean $\lambda+\alpha_1$
will dominate at small $\lambda$, and the smallest intersection
point may occur with a term $\lambda(\maxlength-1)/\maxlength +\alpha'_1$. Indeed, the
difference between coefficients is precisely the smallest possible
$1/\maxlength$, and the difference $|\alpha_1-\alpha'_1|$ may be up to
$2\Delta(A,B)$. This yields the bound $-2\maxlength\mytimes\Delta(A,B)$. An
analogous argument follows when $\lambda$ is large
and $B$ does not have $\bzero$ columns.\\
3: When $A$ and $B$ are real and $\lambda<\underline{C}(A,B)$, all
coefficients in the min-max function $\lambda+A^{\diez}B$ are real
negative, and all coefficients in the min-max function $-\lambda
+B^{\diez}A$ are real positive. This implies that $s(\lambda)$ is
equal to the minimum over $\rad(K_{\lambda}^{(p)})$, which is equal
to $\lambda+\rad(A^{\diez}B)$. An analogous argument follows when
$\lambda>\overline{C}(A,B)$.\eproof
\end{proof}

In Proposition~\ref{amlbml} we will show by an explicit construction
that any slope $k/l$ can be realized as asymptotics of a spectral
function.

We next observe that the asymptotics of $s(\lambda)$ can be read off
from the spectral function $s^{\circ}(\lambda)$, which we introduce below.
For arbitrary $C=(c_{ij})\in\Rmax^{m\times n}$ define
\begin{equation}
\label{abcirc}
c_{ij}^{\circ}=
\begin{cases}
0, & \text{if $c_{ij}\in\R$},\\
-\infty, & \text{if $c_{ij}=-\infty$}.
\end{cases}.
\end{equation}
Let $s^{\circ}(\lambda)$ be the spectral function of the eigenproblem
$A^{\circ}x=\lambda+B^{\circ}x$. 
\begin{proposition}
\label{2ndboolean} Suppose that $A,B\in\Rmax^{m\times n}$ and
that $\lambda k_1/l_1$ where $k_1,l_1\geq 0$ (resp. $-\lambda
k_2/l_2$ where $k_2,l_2\geq 0$) is the left (resp. the right)
asymptotic slope of $s(\lambda)$. Then
\begin{equation}
\label{sclambda}
s^{\circ}(\lambda)=
\begin{cases}
\lambda k_1/l_1, & \text{if $\lambda\leq 0$},\\
-\lambda k_2/l_2, & \text{if $\lambda\geq 0$}.
\end{cases}
\end{equation}
\end{proposition}
\begin{proof}
Observe that the representing matrices $H^{(p\circ)}_{\lambda}$ of
\begin{equation}
\label{hlcirc}
h_{\lambda}^{\circ}:=(\lambda+(A^{\circ})^{\sharp}B^{\circ}x)\enspace\wedge\enspace
(-\lambda+(B^{\circ})^{\sharp}A^{\circ}x)
\end{equation}
are in one-to-one correspondence with the representing matrices
$H^{(p)}_{\lambda}$ of $h_{\lambda}$. The finite entries
$H^{\circ(p)}_{\lambda}$ equal to $\pm\lambda$, they are in the same
places and with the same sign of $\lambda$ as in
$H^{(p)}_{\lambda}$. Hence the cycle means in
$H^{\circ(p)}_{\lambda}$ have the same slopes as the corresponding
cycle means in $H^{(p)}_{\lambda}$, but with zero offsets. When
$s(\lambda)=\rad(h_{\lambda})$ is computed by \eqref{flambda2}, the
asymptotics at large and small $\lambda$ is determined by the slopes
only and yields the same expression as for
$s^{\circ}(\lambda)=\rad(h_{\lambda}^{\circ})$.\eproof
\end{proof}

\subsection{Mean-payoff game oracles and reconstruction problems}
\label{ss:MPG}

Here we consider the problem of
identifying all affine pieces that
constitute the spectral function and
computing the whole spectrum of $(A,B)$
in the case when $A$ and $B$ have integer entries.

The result will be formulated in terms of calls to a mean-payoff game oracle
(computing the value of a mean payoff game).
Let us briefly describe what the mean-payoff games are and how they are related to our problem.
For more precise information the reader may consult Akian et al.~\cite{AGG-10} and Dhingra, Gaubert~\cite{DG-06}, as well as Bj\"{o}rklund, Vorobyov~\cite{bjorklund} and Zwick, Paterson~\cite{ZP-96}. 

It can be observed that the min-max function~$A^{\sharp}B$ is also a dynamic operator of a zero-sum deterministic mean-payoff game, which also corresponds to the system $Ax\leq Bx$. A schematic example of such a game is given in Figure~\ref{f:MPG}, left. Two players, named Max and Min, move a pawn on a bipartite digraph, whose nodes belong either to Max ($\square$) or to Min ($\bigcirc$). In the beginning of the game, the pawn is at a node $j$ of Min, and she has to move it to a node $i$ of Max, paying to him $-a_{ij}$ (some real number). Then Max has to choose a node $k$ of Min. While moving the pawn there, he receives $b_{ik}$ from her. The game proceeds infinitely long, and the aim of Max (resp. Min) is to maximize (resp. minimize) the average payment per turn (meaning a pair of consecutive moves of Min and Max). It turns out that the game has a value, which depends on the starting node of Min. Moreover $\rad(A^{\sharp}B)$ equals the greatest value over all starting nodes (i.e., all nodes of Min).

\begin{figure}
\begin{tabular}{ccccc}
\begin{tikzpicture}[auto, line width=1.5pt, shorten >=1pt,->]
\tikzstyle{vertex1}=[rectangle,fill=black!20,minimum size=15pt,inner sep=1pt]
\tikzstyle{vertex2}=[circle,fill=black!50,minimum size=18pt,inner sep=1pt]
\node[vertex2,xshift=-1.5cm,yshift=1.5cm] (j) {$j$};
\node[vertex2,xshift=-1.5cm,yshift=0cm] (k) {$k$};
\node[vertex2,xshift=-1.5cm,yshift=-1.5cm] (x) {};

\node[vertex1,xshift=1.5cm,yshift=1.5cm] (i) {$i$};
\node[vertex1,xshift=1.5cm,yshift=0cm] (y) {};
\node[vertex1,xshift=1.5cm,yshift=-1.5cm] (z) {};

\draw (j) to node {$-a_{ij}$} (i);
\draw (i) to node {$b_{ik}$} (k);
\draw (k) to (z);
\draw (z) to (x);
\draw (x) to (y);
\end{tikzpicture}

&&&

\begin{tikzpicture}[auto, line width=1.5pt, shorten >=1pt,->, bend right]
  \tikzstyle{vertex1}=[rectangle,fill=black!20,minimum size=30pt,inner sep=1pt]
\tikzstyle{vertex2}=[circle,fill=black!50,minimum size=35pt,inner sep=1pt]
\node[vertex1,xshift=0cm,yshift=1.5cm] (1) {$[m]$};
\node[vertex2,xshift=3cm,yshift=1.5cm] (2) {$[n]$};
\node[vertex1,xshift=6cm,yshift=1.5cm] (3) {$[m]$};
\draw (1) to node [swap] {$\lambda+B$} (2);
\draw (2) to node [swap] {$A^{\sharp}$} (1);
\draw (2) to node [swap] {$-\lambda+B^{\sharp}$} (3);
\draw (3) to node [swap] {$A$} (2);
\end{tikzpicture}
\end{tabular}
\caption{General mean-payoff game (left) and mean-payoff game corresponding to 
$Ax=\lambda+Bx$ (right)}
\label{f:MPG}
\end{figure}
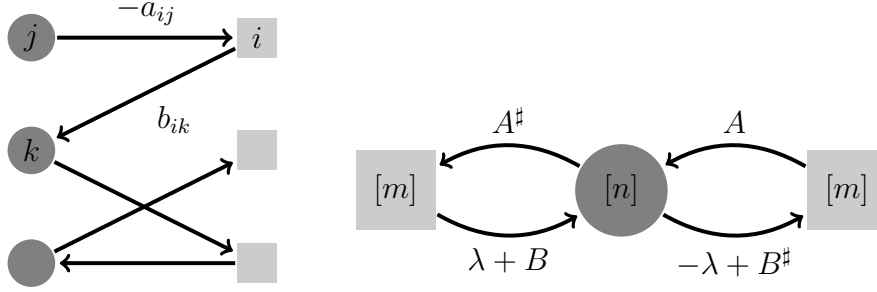

The two-sided eigenproblem $Ax=\lambda+Bx$ can be represented as
\begin{equation}
\label{axlbx-repa}
\begin{pmatrix}
A\\
\lambda+B
\end{pmatrix}  
x\leq
\begin{pmatrix}
\lambda+B\\
A
\end{pmatrix}
x.
\end{equation}

This is equivalent to $x\leq h_{\lambda}(x)$ where 
$h_{\lambda}(x):=(\lambda+A^{\sharp}Bx)\wedge(-\lambda+B^{\sharp}Ax)$ 
as above. Hence the problem $Ax=\lambda+Bx$ corresponds to a parametric mean-payoff game of special kind, with 
$2m$ nodes of Max and $n$ nodes of Min, whose scheme is
displayed on Figure~\ref{f:MPG}, right, where individual nodes of the players are merged
in three large groups.

Denoting by $\MPG(m,n,M)$ the worst-case execution time of any mean-payoff 
oracle computing  $\rad(A^{\sharp}B)$, where $A,B\in\RR^{m\times n}$ have $-\infty$ entries and integer entries with the
greatest absolute value $M$, we immediately obtain that for the same $A$ and $B$ we can find
$s(0)=\rad(h)$ by calling that oracle, in no more than $\MPG(2m,n,M)$ operations.

The implementation of a mean-payoff oracle 
can rely on the policy iteration algorithm of~\cite{CGG-99,DG-06}, as well as
the subexponential algorithm of~\cite{bjorklund} or
the value iteration of~\cite{ZP-96}. Zwick and Paterson~\cite{ZP-96} 
showed that $\MPG(m,n,M)$ is pseudo-polynomial. We use this result below to demonstrate
that the graph of spectral function $s(\lambda)$ can be reconstructed in pseudo-polynomial time.

\begin{theorem}
\label{p:simplecomp}
Let $A,B\in\RR^{m\times n}$ have
only $\bzero$ entries and integer entries with absolute value bounded by $M$. Denote
$\maxlength:=\min(2m,n)$.
\begin{itemize}
\item[1.] All affine pieces that constitute
the function $s(\lambda)$ and hence the spectrum
of $(A,B)$ can be identified in no more than
$\Delta(A,B)\mytimes O(\maxlength^6)$ calls to the mean-payoff game oracle, whose worst-case complexity
is $\MPG(2m,n,\maxlength^2(M+4M\maxlength^2))$. In particular, the reconstruction can be done in 
pseudo-polynomial time.
\item[2.] When $A$ and $B$ have no $\bzero$ columns,
the number of calls needed to reconstruct
the function $s(\lambda)$ can be decreased to
$\Delta(A,B)\mytimes O(\maxlength^5)$, where each call takes no more than
$\MPG(2m,n,\maxlength^2(M+4M\maxlength)$ operations. When $A$ and $B$ are real, the number of calls is decreased to
$(\overline{C}(A,B)-\underline{C}(A,B))\mytimes O(\maxlength^4)$, and the complexity of each call to
$\MPG(2m,n,3M\maxlength^2)$ operations.

\end{itemize}
\end{theorem}
\begin{proof}
In all cases we have a finite interval $L$ of reconstruction, determined by the asymptotics of
$s(\lambda)$. Using Theorem~\ref{t:asymp} , we obtain that in case 1 this is
\begin{equation}
\label{lcase1}
L:=[-2\maxlength^2\Delta(A,B),2\maxlength^2\Delta(A,B)]\subseteq [-4\maxlength^2M, 4\maxlength^2M],
\end{equation}
In case 2, this is
\begin{equation}
\label{lcase21}
L:=[-2\maxlength\Delta(A,B),2\maxlength\Delta(A,B)]\subseteq [-4\maxlength M,4\maxlength M]
\end{equation}
when $A$ and $B$ do not have $\bzero$ columns, or
\begin{equation}
\label{lcase22}
L:=[\underline{C}(A,B),\overline{C}(A,B)]\subseteq [-2M,2M]
\end{equation}
when $A$ and $B$ do not have $\bzero$ entries.

We first compute the asymptotic slopes of $s(\lambda)$ outside $L$.
By Proposition~\ref{2ndboolean}, we can do this by computing
$s^{\circ}(\pm 1)$ in just two calls to the oracle which computes it in no more than
$\MPG(2m,n,1)$ operations. Then the goal is
to reconstruct all affine pieces which constitute $s(\lambda)$ in
the interval $L$.

The affine pieces of $s(\lambda)$ correspond to the maximal cycle
means in the matrices from the representation of $h_{\lambda}(x)$.
The points where such affine pieces may intersect are given by
\begin{equation}
\frac{a_1+k_1\lambda}{n_1}=\frac{a_2+k_2\lambda}{n_2},
\end{equation}
where all parameters are integers and $1\leq |k_1|,|k_2|,n_1,n_2\leq \maxlength$ by Lemma~\ref{l:maxlength}. This implies
\begin{equation}
\lambda=\frac{a_1n_2-a_2n_1}{k_2n_1-k_1n_2}
\end{equation}
The denominators of these points range from $-\maxlength^2$ to $\maxlength^2$, hence
their number is $|L|\mytimes O(\maxlength^4)$ where $|L|$ is the length of the reconstruction interval $L$.
We  reconstruct the whole spectral function by calculating $s(\lambda)$ at these points,
since there is only one affine piece of $s(\lambda)$ between them.

Using~\eqref{lcase1}, \eqref{lcase21} and \eqref{lcase22} we obtain that the absolute value of the entries of $A$ and $\lambda+B$ at each call does not exceed 
$M+4\maxlength^2M$ in case 1, and $M+4\maxlength M$ or $M+2M$ in case 2. Multiplying the entries of $A$ and 
$\lambda+B$ by the denominator of
$\lambda$ which does not exceed $\maxlength^2$, we obtain a problem with integer costs, where all
maximum cycle means $r(H_{\lambda}^{(p)})$ get multiplied by that denominator, and hence $s(\lambda)$
gets multiplied by that denominator as well. Thus we can solve this mean-payoff game instead of
the initial one. In case 1, the new integer problem can be resolved by the mean-payoff oracle in
$\MPG(2m,n,\maxlength^2(M+4\maxlength^2M))$ operations. In case 2, it takes no more than 
$\MPG(2m,n,\maxlength^2(M+4\maxlength M))$ operations when $A$ and $B$ do not have $\bzero$ columns, and
no more than $\MPG(2m,n,3M\maxlength^2)$ operations when $A$ and $B$ do not have $\bzero$ entries. 
The proof is complete.\eproof
\end{proof}

Since $\spec(A,B)$ is the zero set of $s(\lambda)$, we can identify $\spec(A,B)$
by reconstructing $s(\lambda)$ in the intervals given by Proposition~\ref{specints} or more
generally, Theorem~\ref{t:asymp}. However, the task of reconstructing spectrum of $(A,B)$ as zero-level set is even more simple, by the following arguments.

\begin{theorem}
\label{t:rec-spectrum}
Let $A,B\in\RR^{m\times n}$ have
only integer or $\bzero$ entries.
\begin{itemize}
\item[1.] In general, the identification of $\spec(A,B)$ requires
no more than $M\mytimes O(\maxlength^3)$ calls to the mean-payoff game oracle, whose worst-case complexity
is $\MPG(2m,n,2\maxlength(M+2M\maxlength))$. In particular, $\spec(A,B)$
can be identified in pseudo-polynomial time.
\item[2.] If $A$ and $B$ have no $\bzero$ columns, then
the number of calls to the oracle needed to identify $\spec(A,B)$
does not exceed $(\mysup_i (B^{\diez}A0)_i+\mysup_i(A^{\diez}B0)_i)\mytimes O(\maxlength^2)$, and the complexity of the oracle does not exceed $\MPG(2m,n,6M\maxlength)$ operations.
\end{itemize}
\end{theorem}
\begin{proof}
We have to reconstruct the zero-level set of $s(\lambda)$, within a finite interval $L$ of reconstruction. In case 1, we notice that the intersection of $s(\lambda$ with zero level
can occur only at points with absolute value not exceeding $2M\maxlength$ (since $s(\lambda)$
consists of affine pieces $(a+k\lambda)/l$ where $|a|\leq 2M\maxlength$). Hence in case 1
\begin{equation}
\label{lcase111}
L:=[-2M\maxlength, 2M\maxlength].
\end{equation}
In case 2 we use the bounds of Proposition~\ref{specints}:
\begin{equation}
\label{lcase211}
L:=[-\mysup_i(A^{\diez}B0)_i,\mysup_i (B^{\diez}A0)_i]\subseteq [-2M,2M]
\end{equation}
when $A$ and $B$ do not have $\bzero$ columns. In case 1, we also need to check the asymptotics
of $s(\lambda)$ outside the interval, for which we check $s^{\circ}(\pm 1)=0$ (i.e.,
$s^{\circ}(\pm 1)\geq 0$ which takes no more than
$\MPG(2m,n,1)$ operations).

The absolute value of entries of $A$ and $\lambda+B$ does not exceed $M+2M\maxlength$ in
case 1 and $M+2M$ in case 2. We have to check $s(\lambda)=0$ (i.e., $s(\lambda)\geq 0$) 
at all possible intersections
of affine pieces constituting $s(\lambda)$ with zero, i.e., at the points $\lambda=a/k$ within $L$, such that $a$ and
$k$ are integers and $k\leq\maxlength$. 
We also may have to check $s(\lambda)=0$ for one intermediate point between each pair of neighbouring points $\lambda_1$ and $\lambda_2$ such that $s(\lambda_1)=s(\lambda_2)=0$.
If it holds then $s(\lambda)=0$ holds for
the whole interval, and if it does not then it holds only at the ends. 
Note that such an intermediate point for $a_1/k_1$ and $a_2/k_2$ can be chosen as
$(a_1+a_2)/(k_1+k_2)$ thus leading to $k\leq 2\maxlength$.

Multiplying all the entries by $k$ yields a mean-payoff
game with integer costs, for which we check whether the value is nonnegative. This takes
no more than $\MPG(2m,n,2\maxlength(M+2M\maxlength))$ in Case 1 and
$\MPG(2m,n,2\maxlength\times 3M)$ in Case 2, with the number of calls not exceeding 
$|L|O(\maxlength^2)$.\eproof
\end{proof}

Note that this theorem uses the oracles checking $s(\lambda)\geq 0$, not requiring
to compute the exact value.

We can also formulate a certificate that $\lambda$ is an end (left or right) of a spectral
interval.

\begin{proposition}
\label{p:cert}
Supose that $s(\lambda^*)\geq 0$. Then $\lambda^*$ is the left (resp., the right) end
of an interval of $\spec(A,B)$ if and only if there exists a representing matrix
$H_{\lambda}^{(p)}$ where the weights of all cycles are nonpositive, and the slopes
of all cycles with zero weight are strictly positive (resp., negative).
\end{proposition}
\begin{proof}
Condition $s(\lambda^*)\geq 0$ assures that $\lambda^*\in\spec(A,B)$. We will consider the
left end case, for the right end the argument is similar. 
Recall that $s(\lambda^*):=\rad(h_{\lambda^*}$ admits an inf-representation~\eqref{e:low-env} with
selection property. Since there is only finite number of representing matrices, there
exists $H^{(p)}_{\lambda}$ such that $s(\lambda)=\rad(H^{(p)}_{\lambda}$ for $\lambda\in[\lambda^*-\epsilon,\lambda^*]$ (for the right end, we would consider $\lambda\in[\lambda^*,\lambda^*+\epsilon]$). Then $\lambda^*$ is the left end of a spectral interval if and only if
$\rad(H^{(p)}_{\lambda^*})=0$ but $\rad(H^{(p)}_{\lambda})<0$ for $\lambda\in[\lambda^*-\epsilon,\lambda^*)$. After applying the definition of $H_{\lambda}^{(p)}$ \eqref{flambdarows} and
the maximum cycle mean formula for $\rad(H_{\lambda}^{(p)})$ \eqref{mcmh}, the claim follows.\eproof  
\end{proof}

Observe that the condition in Proposition~\ref{p:cert} can be verified in polynomial time for a
given $H^{(p)}$. Namely, it suffices to compute the maximum cycle mean, identify the {\em critical
subgraph} consisting of all cycles with zero cycle mean, and solve the maximum cycle mean problem
for that subgraph, with the edges weighted by $1$ or $-1$ according to the choice of $\lambda$ or
$-\lambda$ in~\eqref{flambdarows}.

The reconstruction of spectral function has been implemented in MATLAB, also
to generate Figures~\ref{curioustest} and~\ref{morsetest}.

\subsection{Examples of analytic computation}
\label{ss:example}

In this section we consider two particular situations when the
spectral function can be constructed analytically. The first example
shows that any asymptotics $k/l$, where $l=1,\ldots,m$ and $k=l-2t$
for $t=1,\ldots,l$, can be realized. The second example is taken
from \cite{Ser-note-10}, and it shows that any system of intervals
and points on the real line can be represented as spectrum of a
max-plus two-sided eigenproblem.

{\bf Asymptotic slopes.} In our first example we consider pairs of
matrices $A^{m,l}\in\Rmax^{m\times m}, B^{m,l}\in \Rmax^{m\times m}$
with entries in $\{0,-\infty\}$, where $0\leq l\leq \lfloor m\rfloor$. An
intuitive idea is to make some ``exchange'' between the
max-plus identity matrix and some cyclic permutation matrix.
For instance
\begin{equation}
\label{ab62}
A^{6,2}=
\begin{pmatrix}
\cdot & \cdot & \cdot & \cdot & \cdot & 0\\
\cdot & 0 & \cdot & \cdot & \cdot & \cdot\\
\cdot & 0 & \cdot & \cdot & \cdot & \cdot\\
\cdot & \cdot & \cdot & 0 &\cdot & \cdot\\
\cdot & \cdot & \cdot & 0 &\cdot & \cdot\\
\cdot & \cdot & \cdot & \cdot & 0 & \cdot
\end{pmatrix},
B^{6,2}=
\begin{pmatrix}
0 & \cdot & \cdot & \cdot & \cdot & \cdot\\
0 & \cdot & \cdot & \cdot & \cdot & \cdot\\
\cdot & \cdot & 0 & \cdot & \cdot & \cdot\\
\cdot & \cdot & 0 & \cdot & \cdot & \cdot\\
\cdot & \cdot & \cdot & \cdot & 0 & \cdot\\
\cdot & \cdot & \cdot & \cdot & \cdot & 0
\end{pmatrix},
\end{equation}
where the dots denote $-\infty$ entries.

Formally, $A^{m,l}=(a_{ij}^{m,l})$ are defined as matrices with
$\{0,-\infty\}$ entries such that $a_{ij}^{m,l}=0$ for $i=1$ and
$j=m$, or $i=j+1$ where $2l<i\leq m$, or $i=j=2k$ where $1\leq k\leq
l$, or $i=2k+1$ and $j=2k$, where $1\leq k< l$, and
$a_{ij}^{m,l}=-\infty$ otherwise.

Similarly, $B^{m,l}=(b_{ij}^{m,l})$ are defined as matrices with entries in
$\{0,-\infty\}$ such that $b_{ij}^{m,l}=0$ for $i=j$ where
$2l<i\leq m$, or $i=j=2k-1$ where $1\leq k\leq l$, or $i=2k$ and
$j=2k-1$, where $1\leq k\leq l$, and $b_{ij}^{m,l}=-\infty$
otherwise.

\begin{proposition}
\label{amlbml} The spectral function associated with $A^{m,l},
B^{m,l}$ consists of two linear pieces: $s(\lambda)=\lambda\mytimes
(m-2l)/m$ for $\lambda\leq 0$ and $s(\lambda)=-\lambda\mytimes
(m-2l)/m$ for $\lambda\geq 0$.
\end{proposition}
\begin{proof}
Let us introduce yet another matrix
$C^{m,l}(\lambda)=(c_{ij}^{m,l}(\lambda))\in\Rmax^{m\times m}$.
Informally, it is a sum of a $\{0,-\infty\}$ permutation (circulant)
matrix and its inverse, weighted by $\pm\lambda$. This pattern
corresponds to the above mentioned ``exchange'' in the construction
of $A^{m,l}$ and $B^{m,l}$. In particular, \eqref{ab62} corresponds
to
\begin{equation}
\label{c62}
C^{6,2}=
\begin{pmatrix}
\cdot & -\lambda & \cdot & \cdot & \cdot & -\lambda\\
\lambda & \cdot & \lambda & \cdot & \cdot & \cdot\\
\cdot & -\lambda & \cdot & -\lambda & \cdot & \cdot\\
\cdot & \cdot & \lambda & \cdot & \lambda & \cdot\\
\cdot & \cdot & \cdot & -\lambda &\cdot & \lambda\\
\lambda & \cdot & \cdot & \cdot & -\lambda & \cdot
\end{pmatrix}.
\end{equation}
Defining formally, $c^{m,l}_{1,m}=-\lambda$,
$c^{m,l}_{m,1}=\lambda$, and
\begin{equation}
\label{cijdef}
c^{m,l}_{ij}=
\begin{cases}
\operatorname{sign}(i,j)\mytimes\lambda, & \text{if $1\leq i,j\leq m$ and $|j-i|=1$},\\
-\infty, & \text{otherwise},
\end{cases}
\end{equation}
where 
\begin{equation}
\operatorname{sign}(i,j)=
\begin{cases}
1, & \text{$j-1=i\geq 2l$ or $j\pm 1=i=2k\leq 2l$},\\
-1, &\text{$i-1=j\geq 2l$ or $i\pm 1=j=2k\leq 2l$}.
\end{cases}
\end{equation}
Observe that the pairs $(i,j)$ and $(j,i)$ for $j=i+1$ (and also
$(1,m)$ and $(m,1)$) have the opposite sign.

It can be shown that each representing max-plus matrix of the
min-max function
\begin{equation}
\label{hml}
h_{\lambda}^{m,l}(x)=
(\lambda+(A^{m,l})^{\sharp}B^{m,l}x)\wedge
(-\lambda+(B^{m,l})^{\sharp}A^{m,l}x)
\end{equation}
is choosing one of the two entries in each row of
$C^{m,l}(\lambda)$. The matrices can be classified according to this
choice as follows (see \eqref{c62} for example):\\
1. Choose $(m,1)$, and $(i,i+1)$ for $i=1,\ldots, m-1$;\\ 2. Choose
$(1,m)$, and $(i,i-1)$ for $i=2,\ldots,m$;\\
3. Choose both $(m,1)$ and $(1,m)$, or both $(i-1,i)$ and $(i,i-1)$
for some
$i=2,\ldots,n$.\\
The first two strategies give just one matrix each, with the
(maximum) cycle means $\lambda\mytimes (m-2l)/m$ and $-\lambda\mytimes
(m-2l)/m$. The rest of the representing matrices are described by
3., and it follows that their maximum cycle means are always greater
than or equal to $0$. Hence $s(\lambda)=\lambda\mytimes (m-2l)/m\wedge
-\lambda\mytimes (m-2l)/m$.\eproof\end{proof}

{\bf The spectrum of two-sided eigenproblem.} Now we consider an
example of \cite{Ser-note-10}. Let us define $A\in\RR^{2\times 3t}$,
$B\in\RR^{2\times 3t}$:
\begin{equation}
\label{ABdef}
\begin{split}
A&=
\begin{pmatrix}
\ldots & a_i & b_i & c_i &\ldots\\
\ldots & 2a_i & 2b_i & 2c_i & \ldots
\end{pmatrix},\\
B&=
\begin{pmatrix}
\ldots & 0 & 0 & 0 &\ldots\\
\ldots & a_i & c_i & b_i &\ldots
\end{pmatrix},
\end{split}
\end{equation}
where $a_i\leq c_i<a_{i+1}$ for $i=1,\ldots, t-1$, where
$b_i:=\frac{a_i+c_i}{2}$. The following result describes
$\spec(A,B)$.

\begin{theorem}[\cite{Ser-note-10}]
\label{mainres-spec}
With $A,B$ defined by \eqref{ABdef},
\begin{equation}
\label{e:mainres-spec}
\spec(A,B)=\bigcup_{i=1}^t [a_i,c_i].
\end{equation}
\end{theorem}

To calculate $s(\lambda)$, which is a more general task, one can
study the representing matrices like in the previous example.
Another way is to guess, for each $\lambda$, a finite eigenvector of
$P_D P_{C(\lambda)}$ and then $s(\lambda)$ is the corresponding
eigenvalue. By this method we obtained that:
\begin{equation}
\label{e:specf}
\specf(\lambda)=
\begin{cases}
\lambda-a_1, & \text{if $\lambda\leq a_1$},\\
0, & \text{if $a_k\leq\lambda\leq c_k$, $k=1,\ldots,t,$}\\
\max(c_k-\lambda,\lambda-a_{k+1}), &\text{if $c_k\leq\lambda\leq a_{k+1}$, $k=1,\ldots,t-1,$}\\
c_t-\lambda, & \text{if $\lambda\geq c_t$}.
\end{cases}
\end{equation}
More precisely, it can be shown that the following vectors are
eigenvectors of $P_DP_{C(\lambda)}$:
\begin{equation}
\label{ylambda}
y^{\lambda}=
\begin{cases}
\begin{array}{c@{{}\quad{}}ccc}
(0 & a_1 & 0 & a_1),
\end{array} &\text{if $\lambda\leq a_1$},\\
\begin{array}{c@{{}\quad{}}ccc}
(0 & \lambda+b_k-a_k & 0 & \lambda+b_k-a_k),
\end{array} &\text{if $a_k\leq\lambda\leq b_k$, $k=1,\ldots,t,$}\\
\begin{array}{c@{{}\quad{}}ccc}
(0 & c_k & 0 & c_k),
\end{array} &\text{if $b_k\leq\lambda\leq c_k$, $k=1,\ldots,t,$}\\
\begin{array}{c@{{}\quad{}}ccc}
(0 & \lambda & 0 & \lambda),
\end{array} &\text{if $c_k\leq\lambda\leq a_{k+1},$ $k=1,\ldots,t-1,$},\\
\begin{array}{c@{{}\quad{}}ccc}
(0 & c_t & 0 & c_t)^T,
\end{array} &\text{if $\lambda\geq c_t$},
\end{cases}
\end{equation}
with the eigenvalues expressed by \eqref{e:specf}.

We can also conclude that in this case $-\rad(A^{\diez}B)=a_1$ and
$\rad(B^{\diez}A)=c_t$. Indeed, by \eqref{e:specf},
$\specf(\lambda)=\lambda-a_1$ for $\lambda\leq a_1$ and
$\specf(\lambda)=c_t-\lambda$ for $\lambda\geq c_t$. Comparing this
with the result of Theorem \ref{t:asymp}, part 3, we get the claim.

As $a_1$ and $c_t$ are eigenvalues, the last result shows that the
bounds given in Theorem \ref{p:specbounds1} cannot be improved in
general.

\begin{figure}
\centering
\includegraphics[width=0.6\linewidth]{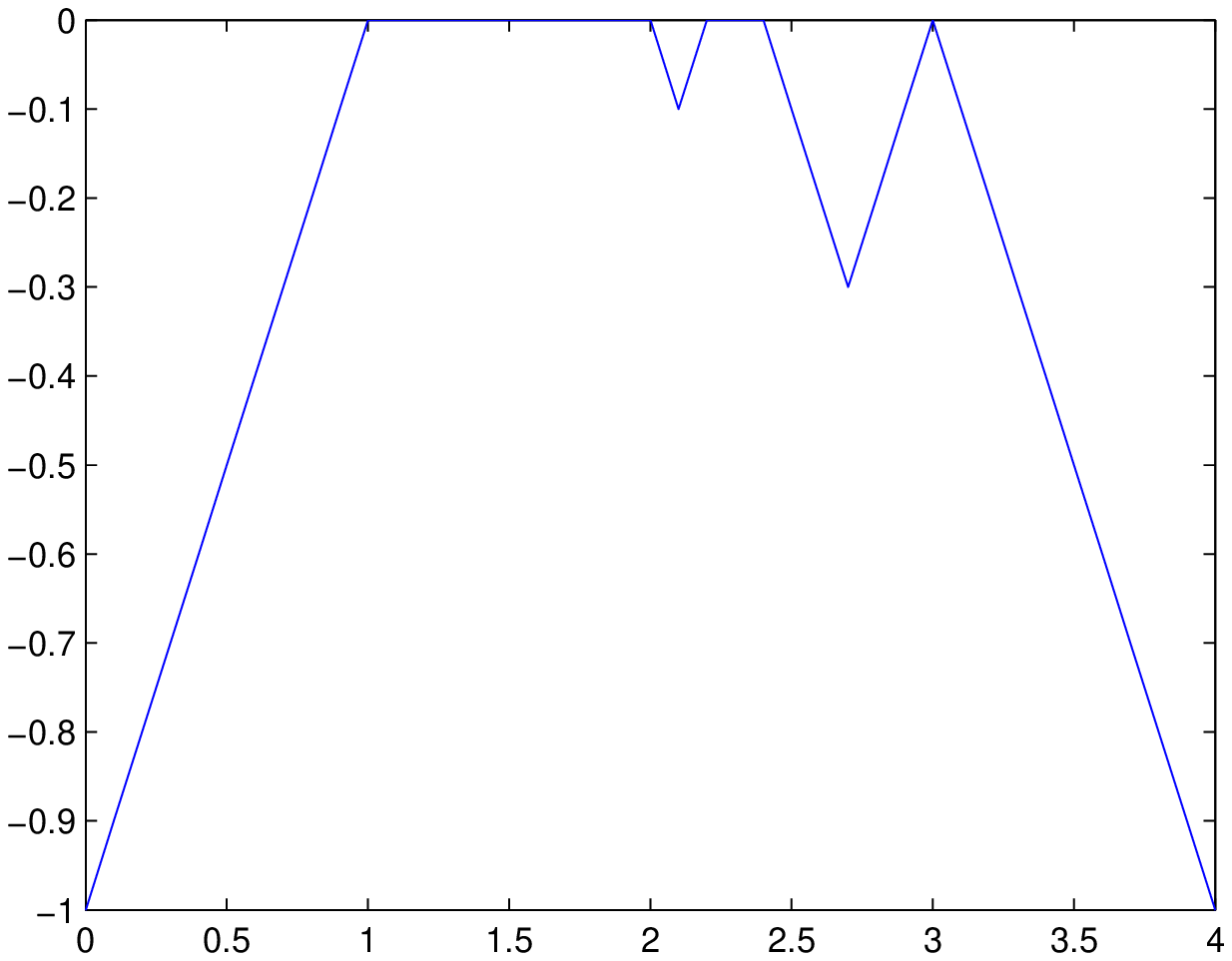}
\caption{The spectral function of $A$ and $B$ in \eqref{e:ABex}}
\label{morsetest}
\end{figure}

For example, take $t=3$, $[a_1,\; c_1]=[1,\; 2]$, $[a_2,\;
c_2]=[2.2,\; 2.4]$ and $[a_3,\; c_3]=[3,\; 3]$. Then
\begin{equation}
\label{e:ABex}
\begin{split}
A&=
\begin{pmatrix}
1 & 1.5 & 2 & 2.2 & 2.3 & 2.4 & 3\\
2 & 3 & 4 & 4.4 & 4.6 & 4.8 & 6
\end{pmatrix},\\
B&=
\begin{pmatrix}
0 & 0 & 0 & 0 & 0 & 0 & 0\\
1 & 2 & 1.5 & 2.2 & 2.4 & 2.3 & 3
\end{pmatrix}
\end{split}
\end{equation}
The spectral function is shown on Figure~\ref{morsetest}.
Note that this is the least Lipschitz function with a given zero-level set. The same observation
holds for the general case~\eqref{e:specf}.

\if{
\subsection{Practical motivations for parametric tropical two-sided systems}
\label{ss:AND/OR}

The two-sided eigenproblem $Ax=\lambda+Bx$ is related to no wait constraints, 
a situation which typically occurs in Gantt charts. Also, suppose that we have to test whether
two different lines, producing the same goods from the same raw materials, can work similarly. 
Then we may ask whether they can produce the goods with the same time differences between
the events of final production. In this case, $\lambda$ means the possible value of time shift, and
$x$ is the vector of starting times. The method of spectral function also yields, for each value of the
time shift $\lambda$, the least Chebyshev distance between the vectors of production times.

More generally, consider tropical systems of inequalities $Ax\leq Bx$. They are basic for our theory, and intimately related
to some problems in scheduling. For instance, a system with AND/OR precedence constraints~\cite{mohring} defined on a set of jobs $[n]:=\{1,\ldots,n\}$
with $m$ waiting conditions $(S,i)$, where $S\subseteq [n]$ and $i\in[n]$.
Waiting conditions $(S,i)$ ($\square$) represent OR-precedence constraints sending further signals as soon as some job $j$ is completed and the required
time separation between that job and $j$ is attained.
Jobs ($\bigcirc$) represent AND-precedence constraints requiring all waiting conditions
to be satisfied. See Figure~\ref{f:quasiMPG}.

Suppose that job $j$ is involved in a waiting condition $(S,i)$ (i.e., $j\in S$).
Denote by $-b_{ij}$ the corresponding time separation. The waiting condition can be 
written as
\begin{equation}
\label{wait-cond}
x_i\geq \min_{j\in S} -b_{ij}+x_j.
\end{equation}
The system of waiting conditions \eqref{wait-cond}, after inverting both sides of the inequalities,
can be written as the system of max-plus linear inequalities $A\otimes x\leq B\otimes x$ where each row of $A$ contains only one finite entry equal to $0$, and $B$ is the matrix with coefficients
$b_{ij}$. Note that each system $A\otimes x\leq B\otimes x$ can be transformed
to this form and understood as a system with AND/OR precedence constraints. 

\begin{figure}
\begin{center}
\begin{tikzpicture}[auto, line width=1.5pt, shorten >=1pt,->]
\tikzstyle{vertex1}=[rectangle,fill=black!20,minimum size=20pt,inner sep=1pt]
\tikzstyle{vertex2}=[circle,fill=black!50,minimum size=20pt,inner sep=1pt]
\node[vertex2,xshift=-1.5cm,yshift=1.5cm] (L1) {AND};
\node[vertex2,xshift=-1.5cm,yshift=0cm] (L0) {AND};
\node[vertex2,xshift=-1.5cm,yshift=-1.5cm] (L-1) {AND};
\node[vertex2,xshift=1.5cm,yshift=0cm] (R) {AND};

\node[vertex1,xshift=0cm,yshift=0.75cm] (C1) {OR};
\node[vertex1,xshift=0cm,yshift=-0.75cm] (C-1) {OR};
\draw (L1) to (C1);
\draw (L0) to (C1);
\draw (L0) to (C-1);
\draw (L-1) to (C-1);
\draw (C1) to (R);
\draw (C-1) to (R);
\end{tikzpicture}
\end{center}
\caption{Jobs as AND-nodes and waiting conditions as OR-nodes}
\label{f:quasiMPG}
\end{figure}

AND-nodes $(\bigcirc)$ and OR-nodes $(\square)$ correspond, respectively,
to the nodes of player Min and the nodes of player Max in the mean-payoff
game associated with the system $Ax\leq Bx$ (described in Subsection~\ref{ss:MPG}). 
The direction of arcs in that mean-payoff game is reversed.

One-parametric extensions of systems with AND/OR
precedence constraints, similar to the two-sided eigenproblem of this paper,  
can be used to handle some optimization problems. See 
Gaubert, Katz and Sergeev~\cite{GKS-11}, where the following example can be found. 
Suppose that, assumed
$x_3=0$, we want to minimize $\min(-1+x_1,-3+x_2)$ over the solutions of
\begin{equation}
\label{max-example}
\begin{split}
& x_1\geq\min(-1+x_2,-1+x_3),\ x_2\geq\min(-1+x_1,-1+x_3),\\
& x_1\geq -2+x_3,\ x_2\geq -2+x_3,\\
& x_1\geq \min(-2+x_2,x_3),
\end{split}
\end{equation} 
which can be interpreted as feasible schedules. Vectors $(-x_1,-x_2,-x_3)$ such that $(x_1,x_2,x_3)$ satisfies~\eqref{max-example} belong to the tropical polyhedron on the left of Figure~\ref{fig-max}, 
and the optimization problem can be solved by reconstructing the spectral function given 
on the right (see~\cite{GKS-11} for precise definition of spectral function in this case).  

We also note here that the method of spectral function carries over to the case when
entries of $A$ and $B$ in $A\otimes x\leq B\otimes x$ are general piecewise-affine functions
of $\lambda$ with integer slopes~\cite{Ser-lastdep}.

\begin{figure}
\begin{center}
\begin{tabular}{c@{{}\qquad{}}c}
\includegraphics[width=0.3\linewidth]{ex-max-poly}& 
\includegraphics[width=0.45\linewidth]{ex-max-sf23}
\end{tabular}
\caption{The tropical polyhedron of~\eqref{max-example}~\cite{GKS-11}, 
and the spectral function.}
\label{fig-max}
\end{center}
\end{figure}
}\fi
 
\section{Conclusions}
We have developed a new approach to the two-sided eigenproblem $A\otimes x=\lambda\otimes B\otimes x$
in max-plus linear algebra, based on parametric min-max functions. This
yields a reduction to mean-payoff games problems, for which a number
of algorithms have already been developed.
We introduced the concept of spectral function $s(\lambda)$, defined as the greatest eigenvalue of the associated parametric min-max function (or the greatest value of the associated mean-payoff game). We showed that $s(\lambda)$ has a natural geometric sense being equal to the inverse of
the least Chebyshev distance between $A\otimes x$ and $B\otimes x$. 
The spectrum of $(A,B)$ can be regarded as the zero-level set of the spectral function, which
is a $1$-Lipschitz function consisting of a finite number of affine pieces. These pieces can be reconstructed in pseudopolynomial time, hence the spectrum of $(A,B)$ can be also effectively
identified. 

A similar approach can be used in max-plus linear programming~\cite{GKS-11}. Spectral functions of a different type are used
in the decision procedure associated with the tropical Farkas lemma in~\cite{AGK-10}, allowing one to check whether a max-plus inequality
can be logically deduced from other max-plus inequalities.
The present approach can be generalized to the case when the entries of $A$ and $B$ are general piecewise-affine functions of $\lambda$~\cite{Ser-lastdep}, but the case of many parameters would be even more interesting. Such development could lead to practical applications in scheduling and design of asynchronous circuits. Also note that the parametric tropical systems are equivalent to parametric mean-payoff games, directing to 
useful stochastic and infinite-dimensional generalizations.  

\section{Acknowledgement} We thank Peter Butkovi\v{c} and Hans Schneider for many useful discussions which have been at the origin of this work.
We are also grateful to the referees for their careful reading and many useful remarks.


\end{document}